\documentclass[11pt,twoside]{amsart}
\usepackage{amsmath,amssymb,amsthm}

\newtheorem{thm}{Theorem}[section]

 \newtheorem{lem}{Lemma}[section]
 \newtheorem{prop}{Proposition}[section]
 \newtheorem{defn}{Definition}[section]
\newtheorem{rem}{Remark}[section]

\def\Id{{\rm Id}\,}
\def\d{\partial}
\def\ddj{\dot \Delta_j}

\def\ddk{\dot \Delta_k}
\def\ddq{\dot \Delta_q}

\def\tilde{\widetilde}
\def\hat{\widehat}
\def\wh{\widehat}
\def\wt{\widetilde}
\def\ov{\overline}

\newcommand{\with}{\quad\hbox{with}\quad}
\newcommand{\andf}{\quad\hbox{and}\quad}

\newcommand\C{\mathbb{C}}
\newcommand\R{\mathbb{R}}

\newcommand\Z{\mathbb{Z}}

\newcommand{\N}{\mathbb{N}}

\newcommand{\Supp}{\hbox{Supp}\,}

\def\sgn{\,\hbox{\rm sgn}\,}

\renewcommand{\div}{\mbox{\rm div}\;\!}

\def\cA{{\mathcal A}}
\def\cB{{\mathcal B}}
\def\cC{{\mathcal C}}

\def\cF{{\mathcal F}}

\def\cP{{\mathcal P}}
\def\cQ{{\mathcal Q}}
\def\cS{{\mathcal S}}
\def\cX{{\mathcal X}}

\newcommand{\Sum}{\displaystyle \sum}

\newcommand{\Int}{\displaystyle \int}

\newcommand{\rhob}{\overline{\varrho}}
\newcommand{\mub}{\overline{\mu}}
\newcommand{\nub}{\overline{\nu}}
\newcommand{\lambb}{\overline{\lambda}}
\newcommand{\cAb}{\overline{\cA}}
\newcommand{\kab}{\overline{\kappa}}
\newcommand{\ppb}{\overline{p}}
\newcommand{\tmu}{\tilde{\mu}}
\newcommand{\tlamb}{\tilde{\lambda}}
\newcommand{\tka}{\tilde{\kappa}}
\newcommand{\Gop}{e^{\sqrt{c_0 t} \Lambda_1}}
\newcommand{\Gopm}{e^{-\sqrt{c_0 t} \Lambda_1}}
\newcommand{\Goptau}{e^{\sqrt{c_0 \tau} \Lambda_1}}

\begin{document}

\title[Compressible Navier-Stokes-Korteweg system]{Gevrey analyticity and decay for the compressible Navier-Stokes system with capillarity
}
\author {Fr\'{e}d\'{e}ric Charve}
\address{Universit\'{e} Paris-Est,  LAMA (UMR 8050), UPEMLV, UPEC, CNRS, 61 avenue du G\'en\'eral de Gaulle, 94010 Cr\'eteil Cedex 10}
\email{frederic.charve@u-pec.fr}
\author{Rapha\"el Danchin}
\address{Universit\'{e} Paris-Est,  LAMA (UMR 8050), UPEMLV, UPEC, CNRS, 61 avenue du G\'en\'eral de Gaulle, 94010 Cr\'eteil Cedex 10}
\email{danchin@univ-paris12.fr}
\thanks{The first two authors are  partially supported by ANR-15-CE40-0011.}

\author{Jiang Xu}
\address{Department of Mathematics,  Nanjing
University of Aeronautics and Astronautics,
Nanjing 211106, P.R.China,}
\email{jiangxu\underline{ }79math@yahoo.com}
\thanks{The third author  is supported by the National
Natural Science Foundation of China (11471158) and the Fundamental Research Funds for the Central
Universities (NE2015005). He also benefited from  a one month invited professor 
position of UPEC, when  this work has been initiated.}


\begin{abstract}
We are concerned with  an isothermal model of viscous and capillary compressible fluids derived by J.~E. Dunn
and J. Serrin (1985), which can be used as a phase transition model. Compared with 
the classical compressible Navier-Stokes equations, there is a smoothing effect on the density 
that comes from the  capillary terms. First,
 we prove  that the global solutions with critical regularity that have been constructed 
in \cite{DD} by the second author and B. Desjardins (2001),  are Gevrey analytic. Second, we extend that result to
a more general critical $L^p$ framework. 
As a consequence, we obtain algebraic time-decay estimates in  critical
Besov spaces (and even exponential decay for the high frequencies) for  any derivatives of the solution.

Our approach  is partly inspired by  the work of   Bae, Biswas \& Tadmor \cite{BBT}
dedicated to the classical incompressible Navier-Stokes equations, 
and requires  our establishing new bilinear estimates (of independent interest) involving the Gevrey regularity for 
the product or composition of functions.

To the best of our knowledge, this is  the first work pointing out Gevrey analyticity for a model of 
\emph{compressible} fluids.
\end{abstract}

\maketitle

\section{Introduction}\setcounter{equation}{0}


When considering a two-phases liquid mixture, it is generally assumed, as a consequence 
of  the Young-Laplace theory, that the phases are separated by a hypersurface and that the jump in the pressure across the hypersurface is proportional to the curvature.

In the  most common description -- the \emph{Sharp Interface} SI model --
the interface between phases corresponds to  a discontinuity in the state space. In contrast,  in the 
\emph{Diffuse Interface} DI model,   the change of phase corresponds 
to a fast but regular transition zone for the density and velocity.

The DI approach has become popular lately as its mathematical and numerical study only requires one set of equations to be solved in a single spatial domain (typically, with a Van der Waals pressure, the phase changes are read through the density values). 
In contrast, with  the SI model one has to solve one system per phase coupled with a free-boundary problem, since the location of the interface is unknown
(see e.g. \cite{CR,Rohdehdr} for more details about the modelling of phase transitions).

\medbreak
The DI model we here aim at considering originates from the works of  Van der Waals and, later, Korteweg
more than one century ago.  The basic idea is to add
 to  the classical compressible fluids equations  a \emph{capillary term},
that  penalizes   high variations of the density. In that way, one  selects only physically relevant solutions, 
that is the ones with density corresponding to either a gas or a liquid, and such that the length of the phase interfaces is minimal.   Indeed, if capillary is absent then one can find an  infinite number of mathematical solutions (most of them being physically wrong although mathematically correct). 
The full derivation of the corresponding equations that we shall  name 
the \emph{compressible Navier-Stokes-Korteweg system} is due to  Dunn and Serrin in \cite{DS}. 
It reads as follows:
\begin{equation}\label{R-E1}
\left\{\begin{array}{l}
\partial_t\varrho+\div(\varrho u)=0,\\[1ex]
\partial_t(\varrho u)+\div(\varrho u\otimes u)-\cA u+\nabla\Pi=\div \mathcal{K}
\end{array}\right.
\end{equation}
where  $\Pi\triangleq P(\varrho)$ is the pressure function, $\cA u \triangleq \div\bigl(2\mu(\rho) D(u)\bigr)+\nabla\big(\lambda (\rho) \div u \big)$ is the diffusion operator, $D(u)=\frac{1}{2}(\nabla u + ^t \nabla u)$ is the symmetric gradient, and the capillarity tensor is given by
$$ \mathcal{K}\triangleq \varrho\,\div(\kappa(\varrho)\nabla\varrho)\,{\rm I}_{\R^d}
+\frac12(\kappa(\varrho)\!-\!\varrho\kappa'(\varrho))|\nabla\varrho|^2\,{\rm I}_{\R^d}-\kappa(\varrho)\nabla\varrho\otimes \nabla\varrho.$$
The density-dependent capillarity function $\kappa$  is assumed to be positive. Note that for smooth enough density and $\kappa,$ we have (see \cite{BDDJ})
\begin{equation}\label{eq:K}
\div\mathcal{K}=\varrho\nabla\Bigl(\kappa(\varrho)\Delta\varrho+\frac12\kappa'(\varrho)|\nabla\varrho|^2\Bigr)\cdotp
\end{equation}
The coefficients $\lambda=\lambda(\varrho)$ and $\mu=\mu(\varrho)$ designate the bulk and shear viscosities, respectively, and are assumed to satisfy in the neighborhood of some reference constant density $\bar\varrho>0$ the conditions
\begin{equation}\label{eq:viscond}
\mu > 0\quad\hbox{and}\quad \nu\triangleq\lambda+2\mu>0.
\end{equation}
Throughout the paper, we shall assume that  the  functions  $\lambda,$ $\mu,$ $\kappa$ and $P,$ 
are real analytic in a neighborhood of $\bar\varrho$. Note that this includes the interesting particular case $\kappa(\rho)=\frac{1}{\rho}$  that corresponds to the so-called quantum fluids. 
The reader may for instance refer to the recent paper by B. Haspot in \cite{H3}
where this case is considered under the `shallow water' assumption for the 
 viscosity coefficients:  $\left(\mu(\rho), \lambda(\rho)\right)=(\ \rho,0).$
\medbreak
System \eqref{R-E1}
is supplemented with initial data
\begin{equation}\label{R-E2}
(\varrho,u)|_{t=0}=(\varrho_0,u_0),
\end{equation}
and we investigate strong solutions in the whole space $\R^d$ with $d\geq2,$ going to a constant equilibrium  $(\bar{\varrho},0)$ with $\bar{\varrho}>0$, at infinity.
\medbreak
The starting point of our paper is the  global existence result  for System \eqref{R-E1} 
in so-called critical Besov spaces
that has been established by the second author and B. Desjardins in \cite{DD}. 
Before stating the result, let us introduce the following  functional space:
$$\displaylines{\qquad E=\Big\{(a,u)\Big|a\in \mathcal{\tilde{C}}_b(\R_+;\dot B^{d/2-1}_{2,1}\cap 
\dot B^{d/2}_{2,1})\cap L^1(\R_+;\dot B^{d/2+1}_{2,1}\cap \dot B^{d/2+2}_{2,1}); \hfill\cr\hfill
 u\in \mathcal{\tilde{C}}_b(\R_+;\dot B^{d/2-1}_{2,1})\cap L^1(\R_+;\dot B^{d/2+1}_{2,1})\Big\},\qquad}$$
 the reader being referred to the appendix for the definition of the Besov spaces
 coming into play in  $E.$
 \medbreak
The following result has been established in \cite{DD}\footnote{Actually,  only the case of constant 
capillarity and viscosity coefficients has been considered therein. 
The case of smooth coefficients may be treated along the same lines
(see also the work by B. Haspot in \cite{H0} concerning the general polytropic case).}:
\begin{thm}\label{thm1.1} 
Let $\bar{\varrho}>0$ be such that $P'(\bar{\varrho})>0$. Suppose that the initial density fluctuation $\varrho_0-\bar{\varrho}$ belongs to
$\dot B^{\frac d2}_{2,1}\cap \dot B^{\frac d2-1}_{2,1}$ and that the initial velocity $u_0$ is in $\dot B^{\frac d2-1}_{2,1}$.

 There exists a constant $\eta>0$ depending only on $\kappa,\mu,\nu,\bar\varrho, P'(\bar{\varrho})$ and $d$, such that, if
$$\|\varrho_0-\bar{\varrho}\|_{\dot B^{\frac d2}_{2,1}\cap \dot B^{\frac d2-1}_{2,1}}+\|u_0\|_{\dot B^{\frac d2-1}_{2,1}}\leq \eta,$$
then System \eqref{R-E1} supplemented with \eqref{R-E2} has a unique global solution 
$(\varrho,u)$ such that $(\varrho-\bar{\varrho},u)\in E.$ \end{thm}

Our first result states  that the solutions constructed in  Theorem \ref{thm1.1} are, in fact, Gevrey analytic. 
\begin{thm}\label{thm3.1}
Let the data $(\rho_0,u_0)$ satisfy the conditions of Theorem \ref{thm1.1} 
for some  $\bar{\varrho}>0$  such that $P'(\bar{\varrho})>0,$
and that the functions $\kappa$, $\lambda,$ $\mu$ and $P$ are analytic.
There exist two positive constants $c_0$ and $\eta$ only depending on those functions and on $d$
  such that if we set 
$$F=\Big\{U\in E\Big| \Gop U\in E\Big\},$$
where  $\Lambda_1$  stands for the Fourier multiplier with
symbol\footnote{Also for technical reasons, as observed before in \cite{L-cras}, it is much more
convenient  to use
the  $\ell^1(\R^d)$ norm rather than the usual $\ell^2(\R^d)$ norm associated with $\Lambda=(-\Delta)^{1/2}$.}   $|\xi|_{1}=\sum_{i=1}^{d}|\xi_{i}|,$ then for any data $(\varrho_0,u_0)$ satisfying 
\begin{equation}\label{eq:data}
\|\varrho_0-\bar{\varrho}\|_{\dot B^{\frac d2}_{2,1}\cap \dot B^{\frac d2-1}_{2,1}}+\|u_0\|_{\dot B^{\frac d2-1}_{2,1}}\leq \eta,
\end{equation}
System  \eqref{R-E1}-\eqref{R-E2} admits a unique solution
$(\varrho,u)$ with  $(\varrho-\bar\varrho,u)\in F$.
\end{thm}
As a by-product,  we shall obtain  time-decay estimates in the critical Besov spaces, for any derivative of the solution (see Theorem \ref{thm5.1} below).
\medbreak
The rest of the paper unfolds as follows. The next section is
devoted to proving Theorem \ref{thm3.1}. 
Then, in Section \ref{sec:4}, we extend  the statement to the critical $L^p$ 
Besov framework. First, we establish a result in the same spirit 
as Theorem \ref{thm1.1}, but in a more general functional framework, then 
we prove that the solutions constructed therein are also 
Gevrey analytic (see Theorem \ref{thm4.1}) and fulfill decay estimates (see Theorem \ref{thm5.1}). 
\medbreak
Before going into the heart of the matter, let us specify some notations. 
Throughout the paper, $C$ stands for a positive harmless ``constant", the meaning of which is clear
 from the context. Similarly,  $f\lesssim g$ means that  $f\leq Cg$ and $f\thickapprox g$ means that $f\lesssim g$ and
$g\lesssim f$.  It will be also understood that  $\|(f,g)\|_{X}\triangleq\|f\|_{X}+\|g\|_{X}$ for all  $f,g\in X$.
Finally, when $f=(f_1,\cdots, f_d)$ with $f_i\in X$ for $i=1,\cdots, d,$  we shall often use, slightly 
abusively, the notation $f\in X$ instead of $f\in X^d.$


\section{The $L^2$ framework}\setcounter{equation}{0}\label{sec:3}

Proving  Theorem \ref{thm3.1} relies essentially 
on  the classical fixed point theorem in the space $F.$ 
To establish that all the conditions are fulfilled however, we need to prove  a couple of
a priori estimates for smooth enough solutions.
To this end, we  first recast the system 
into a more user-friendly shape, then establish  Gevrey 
type estimates for the corresponding linearized system 
about the constant reference state $(\bar\varrho,0),$
and new nonlinear estimates.

\subsection{Renormalization of System   \eqref{R-E1}}

Throughout the paper, it is  convenient to fix some reference viscosity coefficients
$\bar\lambda$ and $\mub,$  pressure $\bar p$  and capillarity coefficient $\bar\kappa,$ 
and to  rewrite  the diffusion, pressure and capillarity terms as follows:
$$
\begin{cases}
\vspace{0.1cm}
\cA u =\mub \div\bigl(2\mu(\rho) D(u)\bigr)+\lambb \nabla\big(\lambda (\rho) \div u \big),\\
\vspace{0.1cm}
\ppb P'(\varrho) \nabla\varrho,\\
\div\mathcal{K}=\kab \varrho\nabla\Bigl(\kappa(\varrho)\Delta\varrho+\frac12\kappa'(\varrho)|\nabla\varrho|^2\Bigr),\end{cases}$$
in such a way that  $\mu(\rhob)=\lambda(\rhob)=\kappa(\rhob)=P'(\rhob)=1$.
\medbreak
 If we denote $\nub=2\mub+\lambb$, then performing the rescaling:
\begin{equation}
\tilde{\varrho} (t,x)=\frac{1}{\rhob}\varrho\Bigl(\frac{\nub}{\rhob \ppb}\,t, \frac{\nub}{\rhob \sqrt{\ppb}}\,x\Bigr),\quad \tilde{u} (t,x)=\frac{1}{\sqrt{\ppb}} u\Bigl(\frac{\nub}{\rhob \ppb}\,t, \frac{\nub}{\rhob \sqrt{\ppb}}\,x\Bigr),
\end{equation}
the parameters $(\rhob, \mub, \lambb, \ppb, \kab)$ are changed into $(1, \frac{\mub}{\nub}, \frac{\lambb}{\nub}, 1, \kab\frac{\rhob^2}{\nub^2})$. We can therefore assume with no loss of generality  that
\begin{equation}
\begin{cases}
 \rhob=1, \quad \nub=2\mub+\lambb=1, \quad \ppb=1,\\
 \mu(1)=\lambda(1)=\kappa(1)=P'(1)=1.
\end{cases}
\end{equation}
Then, introducing the density fluctuation $a=\varrho-1$,   System \eqref{R-E1}  becomes 
\begin{equation}\label{R-E3}
\begin{cases}
 \d_ta+\div u=f,\\
\d_t u-\cAb u+ \nabla a-\kab \nabla\Delta a=g, 
\end{cases}
\end{equation}
with $f=-\div (au)$, and $g=\Sum_{j=1}^5 g_j$, where
\begin{equation}
 \begin{cases}
  \cAb =\mub \div\bigl(2 D(u)\bigr)+\lambb \nabla\big(\div\, u \big) =\mub \Delta u +(\mub +\lambb) \nabla \div u,\\
  g_1 =-u\cdot \nabla u,\\
  g_2 =(1-I(a))\left(2\mub \div \big(\tmu(a) Du\big) + \lambb \nabla \big(\tlamb(a) \div u\big)\right),\\
  g_3 =-I(a) \cAb u,\\
  g_4 =J(a)\cdot \nabla a,\\
  g_5 =\kab \nabla \left( \tka(a) \Delta a+ \frac12 \nabla \tka (a) \cdot \nabla a\right),
 \end{cases}
 \label{fetg}
\end{equation}
and
\begin{equation}
 \begin{cases}
  \tmu(a)=\mu(1+a)-1, \quad \tlamb(a)=\lambda(1+a)-1, \quad \tka(a)=\kappa(1+a)-1,\\
  I(a)=\frac{a}{1+a}, \quad J(a)=1-\frac{P'(1+a)}{1+a}\cdotp
 \end{cases}
\label{fcts}
 \end{equation}
Let us underline  that all those functions are analytic near zero, and vanish at zero.

\subsection{The linearized system}

The present subsection is devoted to exhibiting  the 
smoothing properties of  \eqref{R-E3}, assuming that $f$ and $g$ are given.
In contrast with the linearized equations for the classical compressible Navier-Stokes system, 
we shall see that here both  the density and the velocity are smoothed out  instantaneously. 
The key to that remarkable property 
is given by the following lemma where, as in all this subsection, we  denote by $\wh z$ the Fourier  transform \emph{with respect
to the space variable} of the function $z\in\cC(\R_+;\cS(\R^d))$. 
\begin{lem}\label{lem3.1}
There exist two  positive constants $c_0$  and $C$ depending only on $(\kab,\mub)$
and $\kab,$ respectively,  such that the following inequality holds for all $\xi\in\R^d$ and $t\geq0$:
$$
|(\wh a, |\xi| \wh a,\wh u)(t,\xi)|\leq 
C\biggl(e^{-c_{0}|\xi|^2t}|(\wh a, |\xi| \wh a, \wh u)(0,\xi)|+\int^{t}_{0}e^{-c_{0}|\xi|^2(t-\tau)}|(\widehat{f}, |\xi|\wh{f}, \wh g)(\tau,\xi)|\,d\tau\biggr)\cdotp$$
\end{lem}
\begin{proof} 
 It is mainly a matter of adapting to System \eqref{R-E3} the energy argument of Godunov \cite{G} for partially dissipative first-order symmetric systems (further developed by Kawashima in e.g. \cite{Ka1}).

Note that taking advantage of the Duhamel formula reduces the proof 
to  the case where $f\equiv0$ and $g\equiv0.$
Now, applying  to the second equation of \eqref{R-E3} the Leray projector on divergence free vector fields
yields
\begin{equation}\label{R-E4}
\partial_{t}\cP u-\mub\Delta\cP u=0,
\end{equation}
from which we readily get, after taking the (space) Fourier transform,  
\begin{equation}\label{R-E28}
|\wh{\cP u}(t)|\leq e^{-\mub|\xi|^2t}|\wh{\cP u}(0)|.
\end{equation}

In order to prove the desired inequality for $a$ and the gradient part of the velocity, 
it is convenient to introduce the function  $v\triangleq\Lambda^{-1}\div u$ (with $\Lambda^{s}z\triangleq\mathcal{F}^{-1}\left(\left|\xi\right|^{s}\mathcal{F}z\right)$ for $s\in\R$) .
Then, we discover that $(a,v)$ satisfies (recall that $2\mub+\lambb=1$)
\begin{equation}\label{R-E5}
\left\{\begin{array}{l}\d_ta+\Lambda v=0,\\[1ex]
\d_tv-\Delta v-\Lambda a-\kab \Lambda^3 a=0.
\end{array}\right.
\end{equation}
Hence, taking the Fourier transform of both sides of \eqref{R-E5} gives
\begin{equation}\label{R-E18}
\left\{\begin{array}{l}\frac{d}{dt}\wh a+|\xi|\wh v=0,\\[1ex]
\frac{d}{dt}\wh v+|\xi|^2\wh v-|\xi|(1+\kab|\xi|^2)\wh a=0.
\end{array}\right.
\end{equation}
Multiplying the first equation in \eqref{R-E18} by the conjugate $\overline{\wh a}$ of $\wh a,$ 
and the second one by $\overline{\wh v}$, we get
\begin{equation}\label{R-E19}
\frac12\frac{d}{dt}|\wh a|^2+|\xi| \mathrm{Re}(\wh a\,\ov{\wh v})=0
\end{equation}
and, because  $\mathrm{Re}(\wh a\,\ov{\wh v})=\mathrm{Re}(\ov{\wh a}\,\wh v),$
\begin{equation}\label{R-E20}
\frac12\frac{d}{dt}|\wh v|^2+|\xi|^2|\wh v|^2-|\xi|(1+\kab|\xi|^2)\mathrm{Re}(\wh a\,\ov{\wh v})=0.
\end{equation} 
Multiplying \eqref{R-E19} by  $(1+\kab|\xi|^2),$ and adding up to \eqref{R-E20} yields
\begin{equation}\label{R-E21}
\frac12\frac{d}{dt}\bigl((1+\kab|\xi|^2)|\wh a|^2+|\wh v|^2\bigr)+|\xi|^2|\wh v|^2=0.
\end{equation}
In order to track the dissipation arising for $a,$ let us multiply the first and second equations of \eqref{R-E18} by $-|\xi|\ov{\wh v}$ and $-|\xi|\ov{\wh a},$ respectively. Adding them, we get:
\begin{equation}\label{R-E23}
\frac{d}{dt}\bigl(-|\xi|\mathrm{Re}(\wh a\,\ov{\wh v})\bigr)-|\xi|^3\mathrm{Re}(\ov{\wh a}\wh v)+|\xi|^2(1+\kab|\xi|^2)|\wh a|^2 -|\xi|^2 |\wh v|^2=0.
\end{equation}
Adding to this $|\xi|^2$\eqref{R-E19} yields
\begin{equation}\label{R-E24}
\frac12\frac{d}{dt}\bigl(|\xi|^2 |\wh a|^2-2|\xi|\mathrm{Re}(\wh a\,\ov{\wh v})\bigr)
+|\xi|^2(1+\kab|\xi|^2)|\wh a|^2 -|\xi|^2 |\wh v|^2=0.
\end{equation}
Therefore, by multiplying \eqref{R-E24} by a small enough constant $\beta>0$ (to be determined later) and adding it to \eqref{R-E21}, we get
$$\displaylines{
\qquad\frac12\frac{d}{dt}\mathcal{L}^2_{|\xi|}(t)+\beta|\xi|^2(1+\kab|\xi|^2)|\wh a|^2+(1-\beta)|\xi|^2|\wh v|^2=0,\hfill\cr\hfill
\with\mathcal{L}^2_{|\xi|}(t)\triangleq(1+\kab|\xi|^2)|\wh a|^2+|\wh v|^2+\beta \bigl(|\xi|^2 |\wh a|^2-2|\xi|\mathrm{Re}(\wh a\,\ov{\wh v})\bigr).\qquad}$$

Choosing $\beta=\frac12$ we have $\mathcal{L}^2_{|\xi|}\approx |(\wh a,|\xi| \wh a, \wh v)|^2$ and using the Cauchy-Schwarz inequality, we deduce that
 there exists a positive constant $c_1$ such that on $\R_+,$ we have
$$
\frac{d}{dt}\mathcal{L}^2_{|\xi|}+c_{1}|\xi|^2 \mathcal{L}^2_{|\xi|}\leq 0,
$$
which leads, after time integration, to\footnote{If one tracks the constants
then we get  $c_1=\frac{1}{2} \min(1, \kab)\quad \mbox{and }C=\frac{\max(\frac32,\kab+1)}{\min(\frac12,\kab)}\cdotp$}
\begin{equation}\label{R-E27}
|(\wh a, |\xi| \wh a, \wh v)(t)|\leq Ce^{-c_{1}|\xi|^2t}|(\wh a, |\xi| \wh a, \wh v)(0)|.
\end{equation}
Putting together with \eqref{R-E28} completes the proof of the lemma in the case 
$f\equiv0$ and $g\equiv0.$ The general case readily stems from Duhamel formula.
\end{proof}

We shall also need the following two results that  have been proved in \cite{BBT}.
\begin{lem} \label{lem3.2}
The kernel of operator $M_1:=e^{-[\sqrt{t-\tau}+\sqrt{\tau}-\sqrt{t}]\Lambda_1}$ with $0<\tau < t$
is integrable, and has a  $L^1$ norm that may be bounded independently  of $\tau$ and $t$.
\end{lem}
\begin{lem} \label{lem3.3}
The operator $M_2:=e^{\frac12 a\Delta+\sqrt{a}\Lambda_{1}}$ is a Fourier multiplier which maps boundedly $L^p$ to $L^p$ for all  $1<p<\infty.$ Furthermore, its operator norm is uniformly bounded with respect to $a\geq0$.
\end{lem}

Proving the Gevrey regularity of our solutions will be based on continuity results for the family $(B_t)_{t\geq0}$ of bilinear operators defined by 
$$\begin{aligned}
B_t(f,g)(x)&=\Bigl(\Gop(\Gopm f \cdot \Gopm g)\Bigr)(x)\\ &=\frac1{(2\pi)^{2d}}\int_{\mathbb{R}^d}\int_{\mathbb{R}^d}e^{\mathrm{i}x\cdot (\xi+\eta)}e^{\sqrt{c_0 t}(|\xi+\eta|_1-|\xi|_1-|\eta|_1)}\hat{f}(\xi)\hat{g}(\eta)\,d\xi \,d\eta.\end{aligned}$$
Following \cite{BBT} and \cite{L-book}, we   introduce the following operators acting on functions depending on one real variable:
$${K}_1f\triangleq\frac{1}{2\pi}\int^{\infty}_{0}e^{\mathrm{i}x\xi}\hat{f}(\xi)\,d\xi\andf {K}_{-1}f\triangleq\frac{1}{2\pi}\int^{0}_{-\infty}e^{\mathrm{i}x\xi}\hat{f}(\xi)\,d\xi,$$
and define  $L_{a,1}$ and $L_{a,-1}$ as follows:
$$L_{a,1} f\triangleq f\andf L_{a,-1}f\triangleq
\frac{1}{2\pi}\int_{\mathbb{R}}e^{\mathrm{i}x\xi}e^{-2a|\xi|}\hat{f}(\xi)\,d\xi.$$
For $t\geq0,$ $\alpha=(\alpha_1,\alpha_2,\cdot\cdot\cdot, \alpha_d)$ and $\beta=(\beta_1,\beta_2,\cdot\cdot\cdot, \beta_d)\in \{-1,1\}^{d}$, we set 
$$Z_{t,\alpha,\beta}\triangleq K_{\beta_1}L_{\sqrt{c_0t},\alpha_1\beta_1}\otimes\cdot\cdot\cdot \otimes K_{\beta_d}L_{\sqrt{c_0t},\alpha_d\beta_d}\andf K_{\alpha}\triangleq K_{\alpha_1}\otimes\cdot\cdot\cdot \otimes K_{\alpha_d}.$$
Then we see that 
\begin{equation}
B_t(f,g)=\sum_{(\alpha,\beta,\gamma)\in (\{-1,1\}^d)^3}K_{\alpha}(Z_{t,\alpha,\beta}f Z_{t,\alpha,\gamma}g).
\label{Btsum}
\end{equation}

Since  operators $K_{\alpha}$ and  $Z_{t,\alpha,\beta}$ are linear combinations of smooth homogeneous of degree zero Fourier multipliers, they are bounded on $L^p$ for any $1<p<\infty$ (but they need not
be bounded in $L^1$ and $L^\infty$).  
Furthermore, they commute with all Fourier multipliers 
and thus in particular  with  $\Lambda_1$ and  with  the Littlewood-Paley 
cut-off  operators $\ddj.$
We also have the following fundamental  result:
\begin{lem}\label{lem3.4}
For any $1<p,p_1,p_2<\infty$ with $\frac1p=\frac1{p_1}+\frac1{p_2},$
we have for some constant $C$ independent of $t\geq0,$
$$\|B_t(f,g)\|_{L^p}\leq C\|f\|_{L^{p_1}}\|g\|_{L^{p_2}}.$$
\end{lem}


\subsection{Results of continuity for the paraproduct, remainder and composition}

The aim of this section is to establish the nonlinear estimates 
involving Besov Gevrey regularity that will be needed to bound
the right-hand side of \eqref{R-E3}. We shall actually 
prove  more general estimates both because they are of independent 
interest and since they will be used in the next section, when we shall 
generalize the statement of Theorem \ref{thm3.1} to $L^p$ related Besov spaces.

The first part of this subsection will be devoted to product estimates, and will require 
our using Bony's decomposition and to prove new continuity results for 
the paraproduct and remainder operators. 

Recall that, at the formal level,  the  product  of two  tempered distributions $f$ and $g$   may be decomposed into
\begin{equation}\label{R-E11}
fg=T_fg+T_gf+R(f,g)\end{equation}
with
$$T_fg=\sum_{j\in\Z} \dot{S}_{j-1}f\,\ddj g
\ \hbox{ and }\
R(f,g)=\sum_{j\in\Z}\sum_{|j'-j|\leq1}\ddj f\,\dot\Delta_{j'}g.
$$
The above operators $T$ and $R$ are called ``paraproduct'' and  ``remainder,'' respectively.
The decomposition \eqref{R-E11} has been first introduced
by J.-M. Bony in \cite{Bony}. The paraproduct and remainder operators
possess a lot of continuity properties in Besov spaces (see Chap. 2 in \cite{BCD}), 
which motivates their introduction here.
\medbreak
{}From now on and for notational simplicity, we agree   that   $F(t)\triangleq \Gop f$ for $t\geq0$ 
(and  dependence on $t$  will be often omitted).
\smallbreak
Let us start with paraproduct and remainder estimates in the case where all the Lebesgue indices lie in the range $]1,\infty[.$
\begin{prop}\label{prodlaws1}
\sl{
 Let $s\in \mathbb{R}$ and $1<p<\infty$, $1\leq p_1,p_2,r,r_1,r_2 \leq \infty$ with $1/p=1/p_1+1/p_2$ and $1/r=1/r_1+1/r_2$. If $1<p,p_1,p_2< \infty$, then there exists a constant $C$ such that for any $f,g$ and $\sigma>0$ (or $\sigma\geq 0$ if $r_1=1$),
  \begin{equation}\label{R-E30}
 \|\Gop T_f g\|_{\dot{B}^{s-\sigma}_{p,r}}\leq C \|F\|_{\dot{B}^{-\sigma}_{p_1,r_1}}\|G\|_{\dot{B}^{s}_{p_2,r_2}},
 \end{equation}
 and for any $s_1,s_2 \in \R$ with $s_1+s_2>0$,
\begin{equation}\label{R-E32}
  \|\Gop R(f,g)\|_{\dot{B}^{s_1+s_2}_{p,r}}\leq C \|F\|_{\dot{B}^{s_1}_{p_1,r_1}}\|G\|_{\dot{B}^{s_2}_{p_2,r_2}}.
\end{equation}}
\end{prop}
In order to prove our main results for the Korteweg system, we will need sometimes
the estimates corresponding to the case $p_2=p$ that are contained in  the following statement.
\begin{prop}\label{prodlaws2}
\sl{Assume that $1<p,q<\infty$ and that $1\leq r,r_1,r_2$ fulfill  $1/r=1/r_1+1/r_2$.
There exists a constant $C$ such that for any $f,g$ and $\sigma>0$ (or $\sigma\geq 0$ if $r_1=1$),
  \begin{equation}\label{R-E30b}
 \|\Gop T_f g\|_{\dot{B}^{s-\sigma}_{p,r}}\leq C \|F\|_{\dot{B}^{\frac{d}{q}-\sigma}_{q,r_1}}\|G\|_{\dot{B}^{s}_{p,r_2}},
 \end{equation}
 and for any $s_1,s_2 \in \R$ with $s_1+s_2>0$,
\begin{equation}\label{R-E32b}
  \|\Gop R(f,g)\|_{\dot{B}^{s_1+s_2}_{p,r}}\leq C \|F\|_{\dot{B}^{s_1+\frac{d}{q}}_{q,r_1}}\|G\|_{\dot{B}^{s_2}_{p,r_2}}.
\end{equation}
}
\end{prop}

\begin{proof}[Proof of Proposition \ref{prodlaws1}]
By the  definition of the paraproduct and of $B_t,$ we have 
\begin{equation}\label{estimlocT}
\Gop T_{f}g=\sum_{j\in \mathbb{Z}}W_{j}\with
W_{j}\triangleq B_t(\dot S_{j-1}F,\dot\Delta_jG).
\end{equation}
As no Lebesgue index reaches the endpoints, thanks to Lemma \ref{lem3.4}, we obtain
$$\begin{aligned}
 \|W_{j}\|_{L^p}
 &\lesssim \|\dot{S}_{j-1} F\|_{L^{p_1}} \|\ddj G\|_{L^{p_2}}\\&\lesssim \Big( \sum_{j'\leq j-2}\|\dot{\Delta}_{j'}F\|_{L^{p_{1}}}\Big) \|\ddj G\|_{L^{p_2}}.
  \end{aligned}$$
Therefore, it holds that
$$2^{j(s-\sigma)}\|W_{j}\|_{L^p}\lesssim 2^{js}\|\dot{\Delta}_{j}G\|_{L^{p_{2}}}\sum_{j'\leq j-2}2^{\sigma(j'-j)}2^{-\sigma j'}\|\dot{\Delta}_{j'}F\|_{L^{p_{1}}}.
$$
As $\sigma>0,$ H\"{o}lder and Young inequalities for series enable us to obtain
$$\Big(2^{j(s-\sigma)}\|W_{j}\|_{L^p}\Big)_{\ell^{r}}\lesssim \|F\|_{\dot{B}^{-\sigma}_{p_{1},r_1}}\|G\|_{\dot{B}^{s}_{p_{2},r_2}},$$
and one may conclude to \eqref{R-E30} by using Proposition \ref{prop2.1}.
\smallbreak
In the case $\sigma=0,$ one just has to use the fact that
$$\|\dot S_{j-1}F\|_{L^{p_1}}\lesssim \|F\|_{L^{p_1}}\lesssim 
\|F\|_{\dot B^0_{p_1,1}}.$$

Let us now turn to the remainder: we have for all $k\in \Z$,
$$
 \ddk \Gop R(f,g)=\sum_{j\geq k-2}\sum_{|j-j'|\leq 1}\ddk B_t(\ddj F, \dot\Delta_{j'}G).
$$
Taking the $L^p$ norm with respect to the spatial variable, we deduce by  Lemma \ref{lem3.4} that
$$ \|\ddk \Gop R(f,g)\|_{L^p} 
 \lesssim \sum_{j\geq k-2}\sum_{|j-j'|\leq 1} \|\dot{\Delta}_{j}F\|_{L^{p_{1}}}\|\dot{\Delta}_{j'}G\|_{L^{p_{2}}}.
$$
Then  everything now works as for estimating classical Besov norms:
\begin{multline}\label{R-E34}
2^{k(s_{1}+s_{2})}\|\ddk \Gop R(f,g)\|_{L^p}\\\lesssim \sum_{j\geq k-2}\sum_{|j-j'|\leq 1}
2^{(k-j)(s_{1}+s_{2})}2^{js_1}\|\dot{\Delta}_{j}F\|_{L^{p_{1}}} 2^{(j-j')s_{2}}2^{j's_{2}}\|\dot{\Delta}_{j'}G\|_{L^{p_{2}}},
\end{multline}
and  Young's and H\"{o}lder inequalities for series allow to get \eqref{R-E32} as $s_1+s_2>0$.
\end{proof}
\begin{proof}[Proof of Proposition \ref{prodlaws2}]
We argue as in the previous proof, except that one intermediate
step is needed for bounding the general term of the paraproduct or remainder. 
The key point of course is to bound in $L^p$ the general 
term of $B_t$ in \eqref{Btsum},  while the Lebesgue exponents do not fulfill the conditions of Lemma \ref{lem3.4}.

As an example, let us prove Inequality \eqref{R-E30b} for $\sigma=0.$
We write, combining H\"older and Bernstein inequality \eqref{R-E14}, and the properties
of continuity of operators $K_\alpha$ and $Z_{t,\alpha,\beta},$
$$\begin{aligned}\|K_{\alpha}(Z_{t,\alpha,\beta}\dot S_{j-1} F\, \cdot Z_{t,\alpha,\beta}\ddj G)\|_{L^p}
&\lesssim \|Z_{t,\alpha,\beta}\dot S_{j-1} F\, \cdot Z_{t,\alpha,\beta}\ddj G\|_{L^p}\\
&\lesssim \sum_{j'\leq j-2} \|\dot\Delta_{j'} Z_{t,\alpha,\beta}F\|_{L^\infty}\|Z_{t,\alpha,\beta}\ddj G\|_{L^p}\\
&\lesssim  \sum_{j'\leq j-2} 2^{j'\frac dq}  \|\dot\Delta_{j'}Z_{t,\alpha,\beta}F\|_{L^q}\|Z_{t,\alpha,\beta}\ddj G\|_{L^p}\\
&\lesssim \sum_{j'\leq j-2}2^{j'\frac dq}\|\dot\Delta_{j'}F\|_{L^q}\|\ddj G\|_{L^p}.
\end{aligned}$$
From this, we get 
$$
2^{js} \|W_{j}\|_{L^p} \lesssim \|F\|_{\dot B^{\frac dq}_{q,1}}\, 2^{js}\|\ddj G\|_{L^p}.
$$
We then obtain \eqref{R-E30b} for $\sigma=0$ thanks to Proposition \ref{prop2.1}.
\end{proof}

%
Combining the above propositions with functional  embeddings and Bony's decomposition, 
one may deduce  the following Gevrey product estimates in Besov spaces that  will be of extensive use in what follows:
\begin{prop}\label{prop3.4}
Let $1<p<\infty,$ $s_1,s_2\leq d/p$ with  $s_1+s_2>d\max(0,-1+2/p)$. There exists a constant $C$ such that 
the following estimate holds true:
\begin{equation}\label{R-E36}
\|\Gop (fg)\|_{\dot{B}^{s_1+s_2-\frac dp}_{p,1}}\leq C \|F\|_{\dot{B}^{s_1}_{p,1}}
\|G\|_{\dot{B}^{s_2}_{p,1}}\cdotp
\end{equation}
\end{prop}
\begin{proof}
In light of decomposition \eqref{R-E11}, we have
\begin{equation}\label{eq:bonydec}
\Gop (fg)=\Gop T_{f}g+\Gop T_{g}f+\Gop R(f,g).
\end{equation}
Then \eqref{R-E30b} and standard embedding imply that
$$\begin{cases}
\|\Gop T_{f}g\|_{\dot B^{s_1+s_2-\frac dp}_{p,1}}&\lesssim 
\|F\|_{\dot B^{\frac dp+(s_1-\frac dp)}_{p,1}}\|G\|_{\dot{B}^{s_2}_{p,1}}\\
\|\Gop T_{g}f\|_{\dot B^{s_1+s_2-\frac dp}_{p,1}}&\lesssim 
\|G\|_{\dot{B}^{\frac dp+(s_2-\frac dp)}_{p,1}} \|F\|_{\dot B^{s_1}_{p,1}}.
\end{cases}$$
It is easy to deal with the remainder if $p\geq2$: thanks to embeddings 
and \eqref{R-E32}, we have
$$
\|\Gop R(f,g)\|_{\dot{B}^{s_1+s_2-\frac dp}_{p,1}} \lesssim \|\Gop R(f,g)\|_{\dot{B}^{s_1+s_2}_{p/2,1}}\lesssim \|F\|_{\dot{B}^{s_1}_{p,1}} \|G\|_{\dot{B}^{s_2}_{p,1}}.
$$
If $1<p<2,$ then we use instead that $ \dot B^{\sigma+d(\frac 1{p_0}-\frac 1p)}_{p_0,1}\hookrightarrow
\dot B^{\sigma}_{p,1}$ for all  $1<p_0<p,$ and Inequality \eqref{R-E32} thus implies that
$$\begin{aligned}
 \|\Gop R(f,g)\|_{\dot{B}^{s_1+s_2-\frac dp}_{p,1}} &\lesssim \|\Gop R(f,g)\|_{\dot{B}^{s_1+s_2-2\frac dp +\frac d{p_0}}_{p_0,1}}\\
& \lesssim \|F\|_{\dot{B}^{s_1-\frac{2d}p +\frac d{p_0}}_{p_2,1}}\|G\|_{\dot{B}^{s_2}_{p,1}}\\
& \lesssim \|F\|_{\dot{B}^{s_1-\frac{2d}p +\frac d{p_0} +d(\frac 1p - \frac 1{p_2})}_{p,1}}\|G\|_{\dot{B}^{s_2}_{p,1}}
= \|F\|_{\dot{B}^{s_1}_{p,1}}
\|G\|_{\dot{B}^{s_2}_{p,1}},
\end{aligned}$$
whenever $1/p+1/p_2=1/p_0$, and $p_2 \geq p.$ 
Since $p<2,$ it is clear that those two conditions may be satisfied if taking $p_0$ close enough to $1.$ 
\end{proof}
\begin{rem}\label{r:algebra}
Proposition \ref{prop3.4} ensures that  the space $\bigl\{f\in\dot B^{\frac dp}_{p,1}\,\,\,
e^{\sqrt t\Lambda_1} f\in \dot B^{\frac dp}_{p,1}\bigr\}$  is  an algebra whenever $1<p<\infty.$
\end{rem}
 The previous estimates can be adapted to the Chemin-Lerner's spaces $\tilde{L}^{q}_{T}(\dot{B}^{\sigma}_{p,r})$.
For example, we have the following result. 
\begin{prop}\label{prop3.6}
Let $1< p<\infty$ and $1\leq q,q_1, q_2\leq \infty$ such that $\frac{1}{q}=\frac{1}{q_1}+\frac{1}{q_2}\cdotp$ If 
$\sigma_{1},\sigma_2\leq d/p$ and $\sigma_1+\sigma_2>d\max(0,-1+2/p)$ then there exists a constant $C>0$ such that for all $T\geq0,$ 
\begin{equation}\label{R-E38}
\|\Gop (fg)\|_{\tilde{L}^{q}_{T}(\dot{B}^{\sigma}_{p,1})}\leq C\|F\|_{\tilde{L}^{q_1}_{T}(\dot{B}^{\sigma_1}_{p,1})}
\|G\|_{\tilde{L}^{q_2}_{T}(\dot{B}^{\sigma_2}_{p,1})}\cdotp
\end{equation}
\end{prop}
In order to prove Theorem \ref{thm3.1}, we need not only bilinear estimates involving Gevrey-Besov regularity,
but also composition estimates by real analytic functions.
\begin{lem}\label{l:compo}
\sl{
Let $F$ be a real analytic function in a neighborhood of $0,$ such that $F(0)=0$. Let $1<p<\infty$ and 
$-\min\bigl(\frac dp,\frac d{p'}\bigr)<s\leq\frac dp$ with $\frac 1{p'}=1-\frac 1p\cdotp$ There exist two constants $R_0$ and $D$ depending only on $p,$ $d$ and $F$ such that if for some $T>0$,
\begin{equation}\label{R-E55}
\|\Gop z\|_{\tilde{L}_T^{\infty}(\dot{B}^{\frac{d}{p}}_{p,1})}\leq R_0
\end{equation}
then we have 
\begin{equation}\label{R-E53}
\|\Gop F(z)\|_{{\wt L}^{\infty}_{T}(\dot{B}^{s}_{p,1})}\leq D\|\Gop z\|_{{\wt L}^{\infty}_{T}(\dot{B}^{s}_{p,1})}\cdotp
\end{equation}
}
\end{lem}
\begin{rem} For proving our main results, we shall use the above lemma
with $s=\frac dp$ or $s=\frac dp-1.$ Note that the former case requires that $d\geq2$ and  $1<p<2d.$
\end{rem}

\begin{proof}
Let us write   $$F(z)=\sum^{+\infty}_{n=1}a_{n}z^{n}$$
and denote by $R_F>0$ the convergence radius of the series. For all $t\geq0$ (as usual $Z=\Gop z$) we have
\begin{equation}\label{R-E50}
\Gop F(z)=\sum^{+\infty}_{n=1}a_n\Gop z^n =\sum^{+\infty}_{n=1}a_n \Gop (\Gopm Z)^n,
\end{equation}
which implies from \eqref{R-E38}, by induction and thanks to the condition on $s$, that
$$\begin{aligned}
 \|\Gop F(z)\|_{{\wt L}^{\infty}_{T}(\dot{B}^s_{p,1})}&\leq C \sum^{+\infty}_{n=1} |a_{n}|\, \Big(C\|\Gop z\|_{{\wt L}^{\infty}_{t}(\dot{B}^{\frac dp}_{p,1})}\Big)^{n-1} \|\Gop z\|_{{\wt L}^{\infty}_{T}(\dot{B}^s_{p,1})}\\
& \leq \bar{F}(C\|\Gop z\|_{{\wt L}^{\infty}_{T}(\dot{B}^{\frac dp}_{p,1})}) \|\Gop z\|_{{\wt L}^{\infty}_{T}(\dot{B}^s_{p,1})},
\end{aligned}$$
where we define $\bar{F}(z)=\sum^{+\infty}_{n=1}|a_n|z^{n-1}$.
So when $\|\Gop z\|_{{\wt L}^{\infty}_{T}(\dot{B}^{\frac{d}{p}}_{p,1})}\leq \frac{R_F}{2C} \triangleq R_0$ we have \eqref{R-E53} with $D=\sup_{z\in \bar{B}(0,\frac{R_F}{2})} |\bar{F}(z)|$.
\end{proof}
Let us end this section with a variant of the previous result:
\begin{lem}\label{l:compo2}
\sl{
Let $F$ be a real analytic function in a neighborhood of $0$. Let $1<p<\infty$ and
$-\min\bigl(\frac dp,\frac d{p'}\bigr)<s\leq\frac dp\cdotp$ There exist two constants $R_0$ and $D$ depending only on $p,$ $d$ and $F$ such that if for some $T>0$,
\begin{equation}\label{R-E55b}
\max_{i=1,2}\|\Gop z_i\|_{\tilde{L}_T^{\infty}(\dot{B}^{\frac{d}{p}}_{p,1})}\leq R_0,
\end{equation}
then we have 
\begin{equation}\label{R-E53b}
\|\Gop \big(F(z_2)-F(z_1)\big)\|_{{\wt L}^{\infty}_{T}(\dot{B}^{s}_{p,1})}\leq D\|\Gop (z_2-z_1)\|_{\wt {L}^{\infty}_{T}(\dot{B}^{s}_{p,1})}\cdotp
\end{equation}
}
\end{lem}
\begin{proof}
 With the same notations as before,  Proposition \ref{prop3.6} with $(\sigma_1, \sigma_2)=(s,\frac dp)$  yields:
$$ \begin{aligned}
\|\Gop \big(F(z_2)&-F(z_1)\big)\|_{{\wt L}^{\infty}_{T}(\dot{B}^s_{p,1})} \leq \sum^{+\infty}_{n=1}|a_n|\|\Gop (z_2^n-z_1^n)\|_{{\wt L}^{\infty}_{T}(\dot{B}^s_{p,1})}\\
&\leq \sum^{+\infty}_{n=1}|a_n|\|\Gop \Big((z_2-z_1) \sum_{k=0}^{n-1} z_1^k z_2^{n-1-k}\Big)\|_{{\wt L}^{\infty}_{T}(\dot{B}^s_{p,1})}\\
&\leq C\! \sum^{+\infty}_{n=1}|a_n|\|\Gop (z_2-z_1)\|_{{\wt L}^{\infty}_{T}(\dot{B}^s_{p,1})} \sum_{k=0}^{n-1} \|\Gop (z_1^k z_2^{n-1-k})\|_{{\wt L}^{\infty}_{T}(\dot{B}_{p,1}^{\frac dp})}.
\end{aligned}$$
By induction, we get (using $n\leq 2^{n-1}$)
$$\begin{aligned}
\|&\Gop \big(F(z_2)-F(z_1)\big)\|_{{\wt L}^{\infty}_{T}(\dot{B}^s_{p,1})}\\
&\!\!\leq\! C \|\Gop (z_2-z_1)\|_{{\wt L}^{\infty}_{T}(\dot{B}^s_{p,1})} \sum^{+\infty}_{n=1}|a_n| \sum_{k=0}^{n-1} \!C^{n-1}\|\Gop z_1\|_{{\wt L}^{\infty}_{T}(\dot{B}_{p,1}^{\frac dp})}^k \|\Gop z_2\|_{{\wt L}^{\infty}_{t}(\dot{B}_{p,1}^{\frac dp})}^{n-1-k}\\
&\!\!\leq \! C \|\Gop (z_2-z_1)\|_{{\wt L}^{\infty}_{T}(\dot{B}^s_{p,1})} \sum^{+\infty}_{n=1}|a_n| \Big(2C \max_{i=1,2}\|\Gop z_i\|_{{\wt L}^{\infty}_{T}(\dot{B}_{p,1}^{\frac dp})}\Big)^{n-1}\\
&\!\!\leq \! C \|\Gop (z_2-z_1)\|_{{\wt L}^{\infty}_{T}(\dot{B}^s_{p,1})} \bar{F}\bigl(2C \max_{i=1,2}\|\Gop z_i\|_{{\wt L}^{\infty}_{T}(\dot{B}_{p,1}^{\frac dp})}\bigr)\cdotp
\end{aligned}$$
We conclude as before.
\end{proof}


\subsection{The proof of Theorem \ref{thm3.1}}

One can now come back to the proof of Theorem \ref{thm3.1}. Recall the following estimate that  has been shown in \cite{DD}.
\begin{lem} \label{lem3.5}
Let $(a,u)$ be  a  solution in $E$ of System \eqref{R-E3}. There exists a constant $R_0>0$ 
such that if 
$$\|a\|_{L^\infty(\dot B^{\frac d2}_{2,1})}\leq R_0$$
then  one has the following a priori estimate:
\begin{equation}\label{R-E39}
\|(a,u)\|_{E}\lesssim \|a_0\|_{\dot{B}^{\frac{d}{2}-1}_{2,1}\cap\dot{B}^{\frac{d}{2}}_{2,1}}+\|u_0\|_{\dot{B}^{\frac{d}{2}-1}_{2,1}}+ \big(1+\|(a,u)\|_{E}\big) \|(a,u)\|_{E}^2.
\end{equation}
\end{lem}
We want to generalize it in the Gevrey regularity setting, getting the following result:
\begin{lem} \label{lem3.6}
Let $(a,u)$ be  the  global solution constructed in Theorem \ref{thm1.1}.
Denote $A\triangleq \Gop a$ and $U\triangleq \Gop u$
where $c_0$ is the constant of Lemma \ref{lem3.1}.
 There exists a constant $R_0 >0$ 
such that if 
\begin{equation}\label{eq:smalla}
\|A\|_{\wt L^\infty(\dot B^{\frac d2}_{2,1})}\leq R_0,
\end{equation}
then we have
\begin{equation}\label{R-E40}
\|({A},{U})\|_{E}\lesssim \|a_0\|_{\dot{B}^{\frac{d}{2}-1}_{2,1}\cap\dot{B}^{\frac{d}{2}}_{2,1}}+\|u_0\|_{\dot{B}^{\frac{d}{2}-1}_{2,1}}+ \big(1+\|(A,U)\|_{E}\big) \|(A,U)\|_{E}^2.
\end{equation}
\end{lem}
\begin{proof}
Apply $\ddq $ to \eqref{R-E4}-\eqref{R-E5} and repeat the procedure leading to Lemma \ref{lem3.1}. 
Multiplying by the factor $e^{\sqrt{c_0 t}|\xi|_1}$ we end up with
$$\displaylines{
 |(\widehat{\ddq A},\widehat{\ddq \nabla A}, \widehat{\ddq U})(t,\xi)|\leq C\Big( e^{\sqrt{c_0 t}|\xi|_1} e^{-c_0|\xi|^2t}|(\widehat{\ddq a_0},\widehat{\ddq \nabla a_0},\widehat{\ddq u_0})|\hfill\cr\hfill
 +e^{\sqrt{c_0 t}|\xi|_1} \int^t_0e^{-c_0|\xi|^2(t-\tau)}|(\widehat{\ddq f},\widehat{\ddq \nabla f},\widehat{\ddq }g)|(\tau, \xi)\,d\tau \Big)\cdotp}$$
 Taking the $L^2$ norm, thanks to the Fourier-Plancherel theorem, we get for all $t\geq0,$
 $$\displaylines{
 \|(\ddq {A},\ddq \nabla{A}, \ddq {U})(t)\|_{L^2} \lesssim \|e^{\sqrt{c_0 t}\Lambda_1 +\frac 12 c_0 t \Delta} e^{\frac 12c_0 t \Delta}(\ddq a_0,\ddq \nabla a_0,\ddq u_0)\|_{L^2}\hfill\cr\hfill
 +\!\!\int_0^t\!\! \|e^{(\sqrt{c_0 (t-\tau)}\Lambda_1 \!+\!\frac 12 c_0 (t-\tau) \Delta)}
 e^{-\sqrt{c_0}(\sqrt{t-\tau}+\sqrt{\tau}-\sqrt{t})\Lambda_1} e^{\frac 12c_0 (t-\tau) \Delta}(\ddq F,\ddq \!\nabla F,\ddq G)(\tau)\|_{L^2}d\tau}
$$
and thanks to Lemmas \ref{lem3.2} and \ref{lem3.3}, and to the properties 
of localization of $\ddq,$ we obtain, denoting $c_1\triangleq \frac{9}{32} c_0,$
 \begin{multline}
 \|(\ddq {A},\ddq \nabla{A}, \ddq {U})(t)\|_{L^2} \leq C\biggl(e^{-c_1t2^{2q}}\|(\ddq a_0,\ddq \nabla a_0,\ddq u_0)\|_{L^2}\\ +\int_0^t \|e^{-c_1 (t-\tau) 2^{2q}}(\ddq F,\ddq \nabla F,\ddq G)(\tau)\|_{L^2}\,d\tau\biggr)\cdotp
\label{estimloc}
 \end{multline}
Therefore,  multiplying by $2^{q(\frac d2-1)}$ and summing on $q\in\Z,$ we obtain that for all $t\geq 0,$ 
$$\displaylines{\quad
\|({A},{U})\|_{E_t}\triangleq\|(A,\nabla A, U)\|_{\tilde{L}_t^\infty(\dot{B}_{2,1}^{\frac d2 -1})} +c_0 \|(A,\nabla A, U)\|_{L_t^1 (\dot{B}_{2,1}^{\frac d2 +1})}\hfill\cr\hfill
\leq C\Bigl( \|(a_0, \nabla a_0, u_0)\|_{\dot{B}_{2,1}^{\frac d2 -1}} + \|(F, \nabla F, G)\|_{L_t^1(\dot{B}_{2,1}^{\frac d2 -1})}\Bigr)\cdotp\quad}$$
We are left with  estimating  the external force terms as in the classical Besov case, but using the laws suited to Gevrey regularity.

Regarding $F,$ we have thanks to Proposition \ref{prop3.4}
$$\begin{aligned}
 \int_0^t \|F(\tau)\|_{\dot{B}_{2,1}^{\frac d2 -1}} d\tau& \leq\int_0^t \|\Goptau (au)\|_{\dot{B}_{2,1}^{\frac d2}} d\tau\\
 &\leq C\int_0^t \|A\|_{\dot{B}_{2,1}^{\frac d2}} \|U\|_{\dot{B}_{2,1}^{\frac d2}}\,d\tau 
 \leq \|A\|_{L_t^2(\dot{B}_{2,1}^{\frac d2})} \|U\|_{L_t^2(\dot{B}_{2,1}^{\frac d2})}.
\end{aligned}$$
Estimating $\nabla F$ is also based  on Proposition \ref{prop3.4},
after using  that $f=-u\cdot\nabla a-\nabla(a\div u).$ Then one may write that
$$\begin{aligned}
 \int_0^t \|\nabla F(\tau)\|_{\dot{B}_{2,1}^{\frac d2-1}} d\tau &\leq \int_0^t \|\Goptau (u\cdot\nabla a+ a\,\div u)\|_{\dot{B}_{2,1}^{\frac d2}} d\tau\\
 &\leq C\int_0^t \bigl( \|U\|_{\dot{B}_{2,1}^{\frac d2}}\|\nabla A\|_{\dot B_{2,1}^{\frac d2}}
+  \|A\|_{\dot{B}_{2,1}^{\frac d2}}\|\div U\|_{\dot B_{2,1}^{\frac d2}}\bigr) d\tau\\
 &\leq C\Bigl(\|U\|_{L_t^2(\dot{B}_{2,1}^{\frac d2})} \|A\|_{L_t^2(\dot{B}_{2,1}^{\frac d2+1})} 
 + \|A\|_{L_t^\infty(\dot{B}_{2,1}^{\frac d2})} \|U\|_{L_t^1(\dot{B}_{2,1}^{\frac d2+1})} \Bigr)\cdotp
\end{aligned}$$
One can now turn to $g$: using Proposition \ref{prop3.4}  with $(s_1,s_2)=(\frac d2-1, \frac d2)$ yields
\begin{equation}
 \int_0^t \|\Goptau g_1\|_{\dot{B}_{2,1}^{\frac d2 -1}} d\tau =\int_0^t \|\Goptau (u\cdot \nabla u)\|_{\dot{B}_{2,1}^{\frac d2 -1}} d\tau \leq C \|U\|_{L_t^2(\dot{B}_{2,1}^{\frac d2})}^2.
\end{equation}
Using the same product law together with Lemma \ref{l:compo}, and under the following condition 
that depends on  the convergence radii of the analytic functions appearing in $g$:
\begin{equation}
 \|\Gop a\|_{\wt L^\infty(\dot B^{\frac d2}_{2,1})}\leq \frac{1}{2C}\min(R_I, R_{\tmu}, R_{\tlamb}, R_{\tka}, R_J),
\end{equation}
we get that 
\begin{equation}
 \int_0^t \|\Goptau g_3\|_{\dot{B}_{2,1}^{\frac d2 -1}} d\tau \leq C \|A\|_{L_t^\infty(\dot{B}_{2,1}^{\frac d2})} \|U\|_{L_t^1(\dot{B}_{2,1}^{\frac d2 +1})}.
\end{equation}
Similarly, we obtain:
\begin{equation}
 \begin{cases}
 \vspace{0.2cm}
  \displaystyle{\int_0^t} \|\Goptau g_2\|_{\dot{B}_{2,1}^{\frac d2 -1}} d\tau \leq C (1+ \|A\|_{L_t^\infty(\dot{B}_{2,1}^{\frac d2})})\|A\|_{L_t^\infty(\dot{B}_{2,1}^{\frac d2})} \|U\|_{L_t^1(\dot{B}_{2,1}^{\frac d2 +1})},\\
  \vspace{0.2cm}
  \displaystyle{\int_0^t} \|\Goptau g_4\|_{\dot{B}_{2,1}^{\frac d2 -1}} d\tau \leq C \|A\|_{L_t^2(\dot{B}_{2,1}^{\frac d2})}^2,\\
  \displaystyle{\int_0^t} \|\Goptau \nabla(\tka(a) \Delta a)\|_{\dot{B}_{2,1}^{\frac d2 -1}} d\tau \leq C
   \|A\|_{L_t^\infty(\dot{B}_{2,1}^{\frac d2})} \|A\|_{L_t^1(\dot{B}_{2,1}^{\frac d2+2})}.
 \end{cases}
\end{equation}
We have to be careful with the second part of $g_5$: as Lemma \ref{l:compo} requires the regularity index to be less than $\frac d2$, we have to rewrite the term into:
$$\begin{aligned}
 \int_0^t \|\Goptau \nabla(\nabla(\tka(a)) \cdot \nabla a)\|_{\dot{B}_{2,1}^{\frac d2 -1}} d\tau &\leq\int_0^t \|\Goptau (\tka'(a) \nabla a \cdot \nabla a)\|_{\dot{B}_{2,1}^{\frac d2}} d\tau\\
 &\leq C (1+\|A\|_{L_t^\infty(\dot{B}_{2,1}^{\frac d2})}) \|A\|_{L_t^2(\dot{B}_{2,1}^{\frac d2+1})}^2.
\end{aligned}$$
Putting all the above estimates together, we conclude the proof of Lemma \ref{lem3.6}.
\end{proof}
\medbreak
Now we are able to complete the proof of Theorem \ref{thm3.1} by means of the fixed point theorem. Let  $W(t)$ be the semi-group associated to the left-hand side of \eqref{R-E3}. According to the standard Duhamel formula, one has
$$\displaylines{
\left(  \begin{array}{c}
    a(t) \\    u(t) \\
  \end{array} \right)=\left(  \begin{array}{c}
   a_{L} \\
    u_{L} \\
  \end{array}
\right)+\int^{t}_{0}W(t-\tau)\left(
                               \begin{array}{c}                                 f(\tau) \\
                                 g(\tau) \\                               \end{array}
                             \right)d\tau
                             \with                          
\left(  \begin{array}{c}
   a_{L} \\
    u_{L} \\
  \end{array}
\right)\triangleq W(t)\left(
  \begin{array}{c}
    a_0 \\
    u_0 \\
  \end{array}
\right)\cdotp}$$

Define the functional $\Psi_{(a_{L},u_{L})}$ in a neighborhood of zero in the space $F$ by
\begin{eqnarray}\label{R-E57}
\Psi_{(a_{L},u_{L})}(\bar{a},\bar{u})=\int^{t}_{0}W(t-\tau)\left(
                                            \begin{array}{c}
                                              f(a_{L}+\bar{a},u_{L}+\bar{u}) \\
                                              g(a_{L}+\bar{a},u_{L}+\bar{u}) \\
                                            \end{array}
                                          \right)d\tau.
                                        \end{eqnarray}
To get the existence part of the theorem, it suffices to show that $\Psi_{(a_{L},u_{L})}$ has a fixed point in $F$. Our procedure is divided into two steps : stability 
of some closed ball $\mathcal{B}(0,r)$ of $F$ by $\Psi_{(a_{L},u_{L})},$ then contraction in that ball. 
As those two properties have been established for  the space $E$ in \cite{DD}, 
we shall concentrate on  proving suitable
bounds for $\Goptau\Psi_{(a_{L},u_{L})}(\bar a,\bar u).$

\subsubsection*{Step 1: Stability of some ball $\mathcal{B}(0,r)$}

We prove that the ball $\mathcal{B}(0,r)$ of $F$ is stable under $\Psi_{(a_{L},u_{L})}$, provided the radius $r$ is small enough. Let $a=a_{L}+\bar{a}$ and $u=u_{L}+\bar{u}$. If the data fulfill \eqref{eq:data},  then from Lemmas \ref{lem3.5}, \ref{lem3.6} and the definition of the space $F,$ we get
\begin{equation}\label{R-E58}
\|\Goptau(a_{L},u_{L})\|_{E} \leq C (\|a_0\|_{\dot{B}^{\frac d2-1}_{2,1}\cap\dot{B}^{\frac d2}_{2,1}}+\|u_0\|_{\dot{B}^{\frac d2-1}_{2,1}}) \leq C \eta,
\end{equation}
and
\begin{equation}\label{R-E59}
\|\Goptau\Psi_{(a_{L},u_{L})}(\bar{a},\bar{u})\|_{E}\leq C\|\Goptau (f,\nabla f,g)\|_{L^1_{t}(\dot{B}^{\frac d2-1}_{2,1})}\cdotp
\end{equation}
Assuming that $r$ is so small that:
\begin{eqnarray}\label{R-E60}
\|\Gop a\|_{L^\infty(\dot{B}^{\frac d2}_{2,1})}\leq \|(a,u)\|_F \leq r\leq \frac{1}{2C}\min(R_I, R_{\tmu}, R_{\tlamb}, R_{\tka}, R_J),
\end{eqnarray}
and also that  $2C\eta\leq r,$ we get
\begin{multline}
 \|\Psi_{(a_{L},u_{L})}(\bar{a},\bar{u})\|_{F}\leq C \|(a_{L}+\bar{a},u_{L}+\bar{u})\|_{F}^2\Big(1+ \|(a_{L}+\bar{a},u_{L}+\bar{u})\|_{F}\Big)\\
 \leq C(C\eta+r)^2(1+C\eta+r) \leq C \frac 94 r^2 \Bigl(1+\frac 32 r\Bigr)\cdotp
 \label{R-E61}
\end{multline}
Finally, choosing $(r,\eta)$ such that
$$
r\leq \min \Big (1,\frac{8}{45C},\frac{1}{2C}\min(R_I, R_{\tmu}, R_{\tlamb}, R_{\tka}, R_J)\Big)\ \ \ \mbox{and}\ \ \ \eta\leq \frac{r}{2C},
$$
 assumption \eqref{R-E60} is satisfied. Hence, it follows from \eqref{R-E61} that 
 $$\Psi_{(a_{L},u_{L})}(\cB(0,r))\subset \cB(0,r).$$

\subsubsection*{Step 2: The contraction property}

Let $(\bar{a}_{1},\bar{u}_{1})$ and $(\bar{a}_{2},\bar{u}_{2})$ be  in $\cB(0,r)$. Denote $a_{i}=a_{L}+\bar{a}_{i}$ and $u_{i}=u_{L}+\bar{u}_{i}$ for $i=1,2$. According to \eqref{R-E57} and Lemmas \ref{lem3.5}, \ref{lem3.6}, we have (as already explained, we  focus on bounds for the Gevrey estimates, the estimates in $E$ are in \cite{DD})
$$
\displaylines{\|\Psi_{(a_{L},u_{L})}(\bar{a}_2,\bar{u}_2)-\Psi_{(a_{L},u_{L})}(\bar{a}_1,\bar{u}_1)\|_{F}\hfill\cr\hfill
=\|\Goptau \Big( f(a_2,u_2)-f(a_1,u_1), \nabla f(a_2,u_2)-\nabla f(a_1,u_1),g(a_2,u_2)-g(a_1,u_1)\Big) \|_{L^1(\dot{B}^{\frac{d}{2}-1}_{2,1})}}$$
where $f$ and $g$ are defined in \eqref{fetg}. All terms are estimated exactly as in the previous step except that we use in addition Lemma \ref{l:compo2}. Let us for example give details for $g_2=(1-I(a))\div (\tmu(a) \cdot\nabla a)$)
(assume that $\wt\lambda=0$ for conciseness):
\begin{multline}
\|g_2(a_2,u_2)-g_2(a_1,u_1)\|_{L^1(\dot{B}^{\frac{d}{2}-1}_{2,1})} \leq \|\big(I(a_2)-I(a_1)\big) \div (\tmu(a_2) \cdot\nabla a_2)\|_{L^1(\dot{B}^{\frac{d}{2}-1}_{2,1})}\\
+\|(1-I(a_1)) \div \big((\tmu(a_2)-\tmu(a_1))\nabla u_2+ \tmu(a_1) \cdot\nabla (u_2-u_1)\big)\|_{L^1(\dot{B}^{\frac{d}{2}-1}_{2,1})}.
\end{multline}
Following the previous computations (together with Lemma \ref{l:compo2} for the first and second terms), we obtain, if $r$ and $\eta$ are small enough, 
$$\begin{aligned}
 \|&\Psi_{(a_{L},u_{L})}(\bar{a}_2,\bar{u}_2)-\Psi_{(a_{L},u_{L})}(\bar{a}_1,\bar{u}_1)\|_{F}\\
 &\leq C \Big(\|(a_1,u_1)\|_F +\|(a_2,u_2)\|_F\Big) \Big(1\!+\! \|(a_1,u_1)\|_F \!+\!\|(a_2,u_2)\|_F\Big) \|(\bar{a}_{2}-\bar{a}_{1},\bar{u}_2-\bar{u}_1)\|_{F}\\
&\leq 4C(r+C\eta)(1+r+C\eta) \|(\bar{a}_{2}-\bar{a}_{1},\bar{u}_2-\bar{u}_1)\|_{F} \\
&\leq \frac14\|(\bar{a}_{2}-\bar{a}_{1},\bar{u}_2-\bar{u}_1)\|_{F}.
\end{aligned}$$
Hence, combining the two steps completes the proof of Theorem \ref{thm3.1}.
\qed

\section{The $L^p$ framework}\setcounter{equation}{0}\label{sec:4}

Our aim here is  to extend Theorem \ref{thm3.1} to more general critical Besov spaces.
Recall that for the  classical compressible Navier-Stokes equations,  the first two authors \cite{CD} and Chen-Miao-Zhang \cite{CMZ}  established a global existence result for small 
 data in $L^p$ type critical Besov spaces. The proofs therein are based on the  study of the paralinearized system combined with a Lagrangian change of coordinates. 
A more elementary method has been proposed afterward by B. Haspot 
in \cite{H2}.  It relies on the introduction of some suitable \textit{effective velocity}
that,  somehow,  allows to uncouple the velocity equation from the mass equation.

In the present section, by combining  Haspot's approach with estimates in the same spirit
as the previous section, we shall not only extend the critical regularity 
result in $L^p$ spaces to the capillary case, but also obtain Gevrey analytic regularity:
\begin{thm}\label{thm4.1} Assume  that
the functions $\kappa$, $\lambda,$ $\mu$ and $P$ are real analytic
and that the condition $P'(\rhob)>0$ is fulfilled.
Let $p\in [2,\min(4,2d/(d-2))]$ with, additionally, $p\not=4$ if $d=2.$
There exists an integer $k_0\in\N$ 
and a real number $\eta>0$ depending only on  the functions  $\kappa$, $\lambda,$ $\mu$ and $P,$   and on $p$ and $d$, such that if one defines the threshold between low and high frequencies as  in \eqref{eq:k0}, 
if    $a_0\in \dot B^{\frac d{p}}_{p,1}$
and $u_0\in \dot B^{\frac d{p}-1}_{p,1}$ with,
    besides,  $(a_0^\ell,u_0^\ell)$  in $\dot B^{\frac d2-1}_{2,1}$ satisfy
    \begin{equation}\label{R-E64}
\cX_{p,0}\triangleq \|(a_0,u_0)\|^\ell_{\dot B^{\frac d2-1}_{2,1}}+\|a_0\|^h_{\dot B^{\frac dp}_{p,1}}
+\|u_0\|^h_{\dot B^{\frac d{p}-1}_{p,1}}\leq \eta,\end{equation}
then \eqref{R-E3}  has a unique global-in-time  solution $(a,u)$  in the space $X_p$ defined by
$$\displaylines{
X_{p}\triangleq\{(a,u)|(a,u)^\ell\in \wt\cC_b(\R_+;\dot B^{\frac d2-1}_{2,1})\cap  L^1(\R_+;\dot B^{\frac d2+1}_{2,1}),
a^h\in \wt\cC_b(\R_+;\dot B^{\frac dp}_{p,1})\cap L^1(\R_+;\dot B^{\frac dp+2}_{p,1}),
\cr u^h\in  \wt\cC_b(\R_+;\dot B^{\frac dp-1}_{p,1})
\cap L^1(\R_+;\dot B^{\frac dp+1}_{p,1})\}\cdotp}
$$ Furthermore, there exists a constant $c_0$ so that $(a,u)$ belongs to the space $$Y_p\triangleq\{(a,u)\in X_p|\Gop (a,u)\in X_p\}\cdotp$$
\end{thm}
\begin{rem}\label{rem4.1}
In the  physical dimensions $d=2,3,$
 Condition \eqref{R-E64} allows us to consider the case  $p>d,$  and  the
velocity regularity exponent $d/p-1$  thus  becomes negative. Therefore, our result applies to \emph{large} highly oscillating initial velocities (see e.g. \cite{CD} for more explanations).
\end{rem}


\subsection{Global estimates in $X_p$ for \eqref{R-E3}}

As in Section 2, the proof of Theorem \ref{thm4.1} is based on the fixed point theorem in complete metric spaces. 
Another important  ingredient is the following endpoint maximal regularity property 
of the heat equation with \emph{complex} diffusion coefficient.
\begin{lem}\label{lem4.1}
Let $T>0,$ $s\in \mathbb{R}$ and $1\leq \rho_2, p,r\leq\infty$. Let $u$ satisfy
\begin{equation}\label{R-E65}
\left\{\begin{array}{l}
\partial_t u-\beta\Delta u=f,\\[1ex]
u|_{t=0}=u_0(x),
\end{array}\right.
\end{equation}
where $\beta\in \mathbb{C}$ is a constant parameter with $\mathrm{Re}\,\beta>0$. Then, there exists  a constant $C$ depending only on $d$ and  such that for all $\rho_1\in [\rho_2,\, \infty]$, one has
\begin{equation}\label{R-E66}
(\mathrm{Re}\,\beta)^{\frac{1}{\rho_1}}\|u\|_{\tilde{L}^{\rho_1}_{T}(\dot{B}^{s+\frac{2}{\rho_1}}_{p,r})}\leq C \Big(\|u_0\|_{\dot{B}^{s}_{p,r}}+(\mathrm{Re}\,\beta)^{\frac{1}{\rho_2}-1}\|f\|_{\tilde{L}^{\rho_2}_{T}(\dot{B}^{s-2+\frac{2}{\rho_2}}_{p,r})}\Big)\cdotp
\end{equation}
\end{lem}
\begin{proof}
 We claim that there exists some absolute constants $c$ and $C$  such that
\begin{equation}\label{R-E67}
\|\ddj e^{\beta t\Delta} z\|_{L^p}\leq Ce^{-c\mathrm{Re}\beta\, t2^{2j}}\|\ddj z\|_{L^p},\qquad t\geq0,\ \ j\in\Z.\end{equation}
Indeed, using a suitable rescaling, it suffices to prove \eqref{R-E67} for $j=0.$ Now, if we fix some smooth function 
$\wt\varphi$ compactly supported away from $0$ and with value $1$ on $\varphi,$ then we may write
$$\begin{aligned}
e^{\beta t\Delta}\dot\Delta_0z&=\cF^{-1}\left(\wt\varphi e^{-\beta t|\cdot|^2}
\hat{\dot\Delta_0z}\right)\\[1ex]
 & =  g_{\beta t} \star \dot\Delta_0z\quad\hbox{with}\quad
 g_{\beta t}(x) \triangleq
(2\pi)^{-d} \Int  e^{ix\cdot\xi}\wt\varphi(\xi)
e^{-\beta t|\xi|^2}d\xi.
\end{aligned}
$$
Then, integrating by parts, we discover that for all $x\in\R^d,$ we have
$$g_{\beta t}(x)=(1+|x|^2)^{-d}\int_{\R^d}e^{ix\cdot\xi}
 (\Id-\Delta_\xi)^d  \Bigl(\wt\varphi(\xi)e^{-\beta t|\xi|^2}\Bigr)d\xi.
$$
Expanding the last term  and using the fact that  
integration may be performed on some annulus, we get 
for some positive constants $c,$ $C$ and $C',$  
$$\|g_{\beta t}\|_{L^1}\leq 
C\|(1+|\cdot|^2)^dg_{\beta t}\|_{L^\infty}\leq
C'e^{-ct\mathrm{Re}\,\beta}.$$
Then, using the convolution 
inequality $L^1\star L^p\rightarrow L^p$ yields \eqref{R-E67}.
From it, we get 
\begin{equation}\label{R-68}
\|\ddj u(t)\|_{L^p}\leq C\biggl(e^{-c \mathrm{Re}\beta\,2^{2j}t}\|\ddj u_0\|_{L^p}+\int^t_{0}e^{-c\mathrm{Re}\beta\,2^{2j}(t-\tau)}\|\ddj f(\tau)\|_{L^p}\,d\tau\biggr)\cdotp
\end{equation}
Then, \eqref{R-E66} follows from exactly the same calculations as in  \cite{BCD}.
\end{proof}

Combining Lemma \ref{lem4.1}  with the low frequency estimates of the previous section
and introducing some suitable effective velocity
will enable us to  get  the following  result.
\begin{lem}\label{lem4.2} There exists some constant $C$ such that for all $t\geq0,$ 
\begin{eqnarray}\label{R-E69}
\cX_{p}(t)\leq C (\cX_{p,0}+\cX_{p}^2(t)+\cX_{p}^3(t)),
\end{eqnarray}
where
\begin{multline}\label{R-E70}
 \cX_{p}(t)\triangleq\|(a,u)\|^{\ell}_{\wt L^\infty_{t}(\dot B^{\frac d2-1}_{2,1})}+\|(a,u)\|^{\ell}_{L^1_{t}(\dot B^{\frac d2+1}_{2,1})}
\\+\|a\|^{h}_{\wt L^\infty_{t}(\dot B^{\frac dp}_{p,1})\cap L^1_{t}(\dot B^{\frac dp+2}_{p,1})}
+\|u\|^{h}_{\wt L^\infty_{t}(\dot B^{\frac dp-1}_{p,1})\cap L^1_{t}(\dot B^{\frac dp+1}_{p,1})}.
\end{multline}
\end{lem}
\begin{proof}
We start from the linearized system \eqref{R-E3}:
$$
\begin{cases}
 \d_ta+\div u=f,\\
\d_t u-\cAb u+ \nabla a-\kab \nabla\Delta a=g. 
\end{cases}
$$
The incompressible part of the velocity fulfills the heat equation
\begin{equation}
 \d_t\cP u-\mub\Delta \cP u=\cP g.
 \label{eq:incompr}
\end{equation}
Hence, using 
the notation $z_j:=\ddj z$ for $z$ in $\cS',$ we see that
 there exists a constant $c>0$ such that we have for all $j\in\Z,$
\begin{equation}\label{R-E71}
\|\cP u_j(t)\|_{L^p}\leq Ce^{-c2^{2j}t}\biggl(\|\cP u_j(0)\|_{L^p}+\int_0^te^{c2^{2j}\tau}\|\cP g_j\|_{L^p}\,d\tau\biggr),
\end{equation}
which leads for all $T>0,$  after summation on $j\geq k_0,$ to 
\begin{equation}\label{R-E72}
\|\cP u\|^{h}_{\tilde{L}^{\infty}_{T}(\dot{B}^{\frac{d}{p}-1}_{p,1})}+
\|\cP u\|^{h}_{L^{1}_{T}(\dot{B}^{\frac{d}{p}+1}_{p,1})}\lesssim \|\cP u_0\|^h_{\dot{B}^{\frac{d}{p}-1}_{p,1}}
+\|\cP g\|^h_{L^1_{T}(\dot{B}^{\frac{d}{p}-1}_{p,1})}.
\end{equation}

To estimate $a$ and $\cQ u,$  following Haspot in \cite{H2},  we introduce the modified velocity
$$v\triangleq \cQ u+(-\Delta)^{-1}\nabla a$$
so that $\div v=\div u-a,$  and discover that, since $\bar\lambda+2\bar\mu=1,$
$$\left\{\begin{array}{l}\d_t\nabla a+ \nabla a + \Delta v=\nabla f,\\[1ex]
\d_tv-\Delta v-\kab\Delta\nabla  a=\cQ g+(-\Delta)^{-1}\nabla f+v-(-\Delta)^{-1}\nabla a.
\end{array}\right.$$
In the Fourier space, the eigenvalues of the associated matrix read (with the convention that
$\sqrt r:=i\sqrt{|r|}$ if $r<0$):
$$\lambda^{\pm}(\xi)=\frac12\Bigl(1+|\xi|^2\pm \sqrt{(1-4\kab)|\xi|^4-2|\xi|^2+1}\Bigr)\cdotp$$
Therefore, in the high frequency regime, we expect that for any $\kab>0,$
the system has a parabolic behavior.
This may be easily justified by considering suitable linear combinations of $v$ and $\nabla a.$
Indeed, for all $\alpha\in\C,$ we have
$$\d_t(v+\alpha\nabla a)-(1-\alpha)\Delta v-\kab\Delta\nabla a+\alpha\nabla a=\alpha\nabla f+\cQ g+(-\Delta)^{-1}\nabla f+v-(-\Delta)^{-1}\nabla a.
$$
Therefore, if we set
$$
w\triangleq v+\alpha\nabla a\with \alpha\ \hbox{ satisfying }\ \alpha=\frac\kab{1-\alpha},
$$
then we have
$$
\d_tw-(1-\alpha)\Delta w=-\alpha \nabla a+\alpha\nabla f+\cQ g+(-\Delta)^{-1}\nabla f+v-(-\Delta)^{-1}\nabla a.
$$
A possible choice is
$$
\alpha=\frac12\bigl(1+\sqrt{1-4\kab}\bigr)\quad\hbox{so that }\
1-\alpha =\frac12\bigl(1-\sqrt{1-4\kab}\bigr).
$$

Obviously, the real part of $1-\alpha$ is positive for any value of $\kab$.
Hence one can take advantage of \eqref{R-E66} and get
\begin{multline}\label{R-E73}\|w\|^{h}_{\tilde{L}^{\infty}_{T}(\dot{B}^{\frac{d}{p}-1}_{p,1})}+
\|w\|^{h}_{L^{1}_{T}(\dot{B}^{\frac{d}{p}+1}_{p,1})}\lesssim  \|w_0\|^{h}_{\dot{B}^{\frac{d}{p}-1}_{p,1}}\\+\|\alpha\nabla f+\cQ g+(-\Delta)^{-1}\nabla f\|^{h}_{L^{1}_{T}(\dot{B}^{\frac{d}{p}-1}_{p,1})}+\|v-\alpha\nabla a-(-\Delta)^{-1}\nabla a\|^{h}_{L^{1}_{T}(\dot{B}^{\frac{d}{p}-1}_{p,1})}.
\end{multline}
Because $\nabla(-\Delta)^{-1}$ is an homogeneous Fourier multiplier of degree $-1,$  we have
$$\begin{aligned}
\|\alpha\nabla f+\cQ g+(-\Delta)^{-1}\nabla f\|^{h}_{L^{1}_{T}(\dot{B}^{\frac{d}{p}-1}_{p,1})}&\lesssim  \|f\|^{h}_{L^{1}_{T}(\dot{B}^{\frac{d}{p}}_{p,1})}+
\|f-\mathrm{div}g\|^{h}_{L^{1}_{T}(\dot{B}^{\frac{d}{p}-2}_{p,1})}\nonumber\\&\lesssim  \|f\|^{h}_{L^{1}_{T}(\dot{B}^{\frac{d}{p}}_{p,1})}+\|g\|^{h}_{L^{1}_{T}(\dot{B}^{\frac{d}{p}-1}_{p,1})}.
\end{aligned}$$
Next, let us observe that, owing to the high frequency cut-off, we have for some
universal constant $C$,
$$\displaylines{
\|\alpha\nabla a\|^{h}_{L^{1}_{T}(\dot{B}^{\frac{d}{p}-1}_{p,1})}\leq C 2^{-2k_{0}}\|a\|^{h}_{L^{1}_{T}(\dot{B}^{\frac{d}{p}+2}_{p,1})},\qquad
\|v\|^{h}_{L^{1}_{T}(\dot{B}^{\frac{d}{p}-1}_{p,1})} \leq C 2^{-2k_{0}}\|v\|^{h}_{L^{1}_{T}(\dot{B}^{\frac{d}{p}+1}_{p,1})}
\cr\andf \|(-\Delta)^{-1}\nabla a\|^{h}_{L^{1}_{T}(\dot{B}^{\frac{d}{p}-1}_{p,1})}\leq C 2^{-4k_{0}}\|a\|^{h}_{L^{1}_{T}(\dot{B}^{\frac{d}{p}+2}_{p,1})}.}$$
Consequently, it follows that
\begin{multline}\label{R-E74}
\|w\|^{h}_{\tilde{L}^{\infty}_{T}(\dot{B}^{\frac{d}{p}-1}_{p,1})}+
\|w\|^{h}_{L^{1}_{T}(\dot{B}^{\frac{d}{p}+1}_{p,1})}\lesssim  \|w_0\|^{h}_{\dot{B}^{\frac{d}{p}-1}_{p,1}}
+\|f\|^{h}_{L^{1}_{T}(\dot{B}^{\frac{d}{p}}_{p,1})}+\|g\|^{h}_{L^{1}_{T}(\dot{B}^{\frac{d}{p}-1}_{p,1})}\\
+2^{-2k_{0}}\|a\|^{h}_{L^{1}_{T}(\dot{B}^{\frac{d}{p}+2}_{p,1})}+2^{-2k_{0}}\|v\|^{h}_{L^{1}_{T}(\dot{B}^{\frac{d}{p}+1}_{p,1})}
+2^{-4k_{0}}\|a\|^{h}_{L^{1}_{T}(\dot{B}^{\frac{d}{p}+2}_{p,1})}.
\end{multline}
Now, in order to estimate $v,$ we use the fact that
\begin{equation}\label{eq:Da}\nabla a =\frac{w-v}\alpha\end{equation} so that the equation for $v$ rewrites
$$
\d_t v-\frac{\alpha-\kab}{\alpha}\:\Delta v=\frac\kab\alpha \Delta w +\nabla(-\Delta)^{-1}(f-\mathrm{div}g)+ v-(-\Delta)^{-1}\nabla a.
$$
The important observation is that
$$
\frac{\alpha-\kab}{\alpha}=\frac\kab{1-\alpha}\cdotp
$$
Hence one can again take advantage of \eqref{R-E66}, and get
$$\displaylines{
\|v\|^{h}_{\tilde{L}^{\infty}_{T}(\dot{B}^{\frac{d}{p}-1}_{p,1})}+
\|v\|^{h}_{L^{1}_{T}(\dot{B}^{\frac{d}{p}+1}_{p,1})}\lesssim\|v_0\|^{h}_{\dot{B}^{\frac{d}{p}-1}_{p,1}}
+\|\nabla(-\Delta)^{-1}(f-\mathrm{div}g)\|^{h}_{L^{1}_{T}(\dot{B}^{\frac{d}{p}-1}_{p,1})}\hfill\cr\hfill+\Big\|\frac\kab\alpha \Delta w+v-(-\Delta)^{-1}\nabla a\Big\|^h_{L^{1}_{T}(\dot{B}^{\frac{d}{p}-1}_{p,1})},}
$$
whence
\begin{multline}\label{R-E75}
\|v\|^{h}_{\tilde{L}^{\infty}_{T}(\dot{B}^{\frac{d}{p}-1}_{p,1})}+
\|v\|^{h}_{L^{1}_{T}(\dot{B}^{\frac{d}{p}+1}_{p,1})}\lesssim\|v_0\|^{h}_{\dot{B}^{\frac{d}{p}-1}_{p,1}}+\|g\|^{h}_{L^{1}_{T}(\dot{B}^{\frac{d}{p}-1}_{p,1})}+\|f\|^{h}_{L^{1}_{T}(\dot{B}^{\frac{d}{p}-2}_{p,1})}\\
+\|w\|^{h}_{L^{1}_{T}(\dot{B}^{\frac{d}{p}+1}_{p,1})}+2^{-2k_{0}}\|v\|^{h}_{L^{1}_{T}(\dot{B}^{\frac{d}{p}+1}_{p,1})}
+2^{-4k_{0}}\|a\|^{h}_{L^{1}_{T}(\dot{B}^{\frac{d}{p}+2}_{p,1})}.
\end{multline}
Plugging  \eqref{R-E74} in \eqref{R-E75} and taking $k_0$ large enough, we arrive at
$$
\|v\|^{h}_{\tilde{L}^{\infty}_{T}(\dot{B}^{\frac{d}{p}-1}_{p,1})}+
\|v\|^{h}_{L^{1}_{T}(\dot{B}^{\frac{d}{p}+1}_{p,1})}\lesssim\|v_0\|^{h}_{\dot{B}^{\frac{d}{p}-1}_{p,1}}+\|f\|^{h}_{L^{1}_{T}(\dot{B}^{\frac{d}{p}}_{p,1})}
+\|g\|^{h}_{L^{1}_{T}(\dot{B}^{\frac{d}{p}-1}_{p,1})}
+2^{-2k_{0}}\|a\|^{h}_{L^{1}_{T}(\dot{B}^{\frac{d}{p}+2}_{p,1})}.
$$
Then, inserting that latter inequality in \eqref{R-E74} and using \eqref{eq:Da}, we get 
$$
\|(\nabla a, v)\|^{h}_{\tilde{L}^{\infty}_{T}(\dot{B}^{\frac{d}{p}-1}_{p,1})}+\|(\nabla a, v)\|^{h}_{L^{1}_{T}(\dot{B}^{\frac{d}{p}+1}_{p,1})}
\lesssim \|(\nabla a_0, v_0)\|^h_{\dot{B}^{\frac{d}{p}-1}_{p,1}}+\|f\|^{h}_{L^{1}_{T}(\dot{B}^{\frac{d}{p}}_{p,1})}
+\|g\|^{h}_{L^{1}_{T}(\dot{B}^{\frac{d}{p}-1}_{p,1})}.$$
Finally, keeping  in mind that $u=v-(-\Delta)^{-1}\nabla a+\mathcal{P}u$,  we conclude that
\begin{multline}
 \|(\nabla a, u)\|^{h}_{\tilde{L}^{\infty}_{T}(\dot{B}^{\frac{d}{p}-1}_{p,1})}+\|(\nabla a, u)\|^{h}_{L^{1}_{T}(\dot{B}^{\frac{d}{p}+1}_{p,1})}
\\\lesssim \|(\nabla a_0, u_0)\|^h_{\dot{B}^{\frac{d}{p}-1}_{p,1}}+\|f\|^{h}_{L^{1}_{T}(\dot{B}^{\frac{d}{p}}_{p,1})}+\|g\|^{h}_{L^{1}_{T}(\dot{B}^{\frac{d}{p}-1}_{p,1})}.
\label{estimHF}
\end{multline}


Let us next turn to estimates for the nonlinear terms. For the high frequencies of $f,$ we just
write that
\begin{eqnarray}\label{R-E78}
\|f\|^{h}_{L^{1}_{T}(\dot{B}^{\frac{d}{p}}_{p,1})}&\!\!\!\lesssim\!\!\!& \|a\|_{L^\infty_{T}(L^\infty)}\|u\|_{L^{1}_{T}(\dot{B}^{\frac{d}{p}+1}_{p,1})}
+\|u\|_{L^2_{T}(L^\infty)}\|a\|_{L^{2}_{T}(\dot{B}^{\frac{d}{p}+1}_{p,1})}\nonumber\\&\!\!\!\lesssim\!\!\! &
\|a\|_{L^\infty_{T}(\dot B^{\frac dp}_{p,1})}\|u\|_{L^{1}_{T}(\dot{B}^{\frac{d}{p}+1}_{p,1})}+\|u\|_{L^2_{T}(\dot B^{\frac dp}_{p,1})}\|a\|_{L^{2}_{T}(\dot{B}^{\frac{d}{p}+1}_{p,1})}
\nonumber\\&\!\!\!\lesssim\!\!\! &\cX_p^2(T).
\end{eqnarray}
All terms in $g,$ but $\nabla(\wt \kappa(a)\Delta a)$ and $\nabla(\tka'(a)|\nabla a|^2)$ have been treated in e.g. \cite{H2}
for the classical compressible Navier-Stokes equations; they are bounded by the right-hand side of \eqref{R-E69}.
Now, regarding the high frequencies of these two capillary  terms, one can just use the fact that the space $\dot B^{\frac dp}_{p,1}$ is stable by product and composition, and  get
$$\|\nabla(\wt\kappa(a)\Delta a)\|^h_{L^1(\dot B^{\frac dp-1}_{p,1})}
\lesssim \|\wt\kappa(a)\Delta a\|_{L^1(\dot B^{\frac dp}_{p,1})}
\lesssim \|a\|_{L^\infty(\dot B^{\frac dp}_{p,1})}
\|\Delta a\|_{L^1(\dot B^{\frac dp}_{p,1})}.$$
Similarly,
\begin{multline}
 \|\nabla(\tka'(a)|\nabla a|^2)\|^h_{L^1(\dot B^{\frac dp-1}_{p,1})}
\lesssim \|\tka'(1)+(\tka'(a)-\tka'(1))|\nabla a|^2\|_{L^1(\dot B^{\frac dp}_{p,1})}\\
\lesssim (1+ \|a\|_{L^\infty(\dot B^{\frac dp}_{p,1})})\|\nabla a \|^2_{L^2(\dot B^{\frac dp}_{p,1})}.
\end{multline}
To handle the low frequencies, one can  use the fact that, owing to Lemma \ref{lem3.1},  
\begin{equation}
\|(a,u)\|_{\wt L_T^\infty(\dot B^{\frac d2-1}_{2,1})\cap L_T^1(\dot B^{\frac d2+1}_{2,1})}^\ell
\lesssim \|(a_0,u_0)\|_{\dot B^{\frac d2-1}_{2,1}}^\ell+\|(f,g)\|_{L_T^1(\dot B^{\frac d2-1}_{2,1})}^\ell.
\end{equation}
Again, taking advantage of prior works on the compressible Navier-Stokes equations, 
we just have to check that the capillary terms satisfy \eqref{R-E69}. 
Now, we have
$$
\|\nabla(\wt \kappa(a)\Delta a)\|^\ell_{L^1(\dot B^{\frac d2-1}_{2,1})}
\lesssim \|\wt \kappa(a)\Delta a\|^\ell_{L^1(\dot B^{\frac d2}_{2,1})}.$$
In order to estimate the r.h.s., we use the following Bony decomposition:
$$\wt \kappa(a)\Delta a =T_{\wt\kappa(a)} \Delta a +T_{\Delta a}\wt\kappa(a)+ R(\wt\kappa (a), \Delta a).$$
Recall that $T: \dot B^{\frac dp-1}_{p,1}\times \dot B^{\frac dp}_{p,1}\to \dot B^{\frac d2-1}_{2,1}$
for $2\leq p \leq\min(4,\frac{2d}{d-2})\cdotp$ Hence we have

$$\begin{aligned}\|T_{\wt\kappa(a)} \Delta a +T_{\Delta a}\wt\kappa(a)\|^\ell_{\dot B^{\frac d2}_{2,1}}
&\lesssim \|T_{\wt\kappa(a)} \Delta a +T_{\Delta a}\wt\kappa(a)\|_{\dot B^{\frac d2-1}_{2,1}}\\
&\lesssim  \|\wt\kappa(a)\|_{\dot B^{\frac dp-1}_{p,1}}\|\Delta a\|_{\dot B^{\frac dp}_{p,1}}
+ \|\Delta a\|_{\dot B^{\frac dp-1}_{p,1}}\|\wt\kappa(a)\|_{\dot B^{\frac dp}_{p,1}}\\
&\lesssim  \|a\|_{\dot B^{\frac dp-1}_{p,1}}\|\Delta a\|_{\dot B^{\frac dp}_{p,1}}.\end{aligned}$$
For the remainder term, one can use that $R:\dot B^{\frac dp}_{p,1}\times\dot B^{\frac dp}_{p,1}\to\dot B^{\frac d2}_{2,1}$
if $2\leq p\leq 4.$
Hence we eventually get, if $p\leq\min(4,\frac{2d}{d-2}),$
$$\|\nabla(\wt \kappa(a)\Delta a)\|^\ell_{L^1(\dot B^{\frac d2-1}_{2,1})}\lesssim  \|a\|_{L^\infty(\dot B^{\frac dp}_{p,1}\cap \dot B^{\frac dp-1}_{p,1})}\|\Delta a\|_{L^1(\dot B^{\frac dp-1}_{p,1})}.$$
In order to estimate the other capillary term, we simply use that $\tka'(a)\nabla a= \nabla(\tka(a))$, with $\tka(0)=0$. 
Now, thanks to  Bony's decomposition:
$$\nabla a \cdot \nabla(\tka(a)) = T_{\nabla a} \nabla(\tka(a))+ T_{\nabla(\tka(a))} \nabla a
+R(\nabla(\tka(a)),\nabla a).$$
and to
$$\begin{aligned}
 \| T_{\nabla a} \nabla(\tka(a))\!+\! T_{\nabla(\tka(a))} \nabla a
\|^\ell_{\dot B^{\frac d2}_{2,1}}&\lesssim\| \nabla a\|_{\dot B^{\frac dp-1}_{p,1}}\| \nabla(\tka(a))\|_{\dot B^{\frac dp}_{p,1}}
+ \| \nabla(\tka(a))\|_{\dot B^{\frac dp-1}_{p,1}}\| \nabla a\|_{\dot B^{\frac dp}_{p,1}}\\
&\lesssim\|a\|_{\dot B^{\frac dp}_{p,1}} \| \nabla a\|_{\dot B^{\frac dp}_{p,1}}
\end{aligned}
$$
and
$$
\| R(\nabla(\tka(a)),\nabla a)\|_{\dot B^{\frac d2}_{2,1}} \lesssim
 \| \nabla(\tka(a))\|_{\dot B^{\frac dp}_{p,1}}\| \nabla a\|_{\dot B^{\frac dp}_{p,1}}
\lesssim   \|a\|_{\dot B^{\frac dp+1}_{p,1}}\| \nabla a\|_{\dot B^{\frac dp}_{p,1}},$$
we end up with
$$
\|\nabla(\kappa'(a)|\nabla a|^2)\|^\ell_{L^1(\dot B^{\frac d2-1}_{2,1})}
\lesssim   \|a\|_{L^2(\dot B^{\frac dp}_{p,1}\cap \dot B^{\frac dp+1}_{p,1})}\|\nabla a\|_{L^2(\dot B^{\frac dp}_{p,1})}.$$

Combining with the already proved estimates for the other nonlinear terms (see \cite{D-handbook}), we conclude that \eqref{R-E69} is fulfilled.
From this,  it is not difficult to work out a fixed  point argument as in the previous section,
and to prove the first part of Theorem \ref{thm4.1}.
\end{proof}

\medbreak

\subsection{More paraproduct, remainder and  product estimates}

In order to investigate the Gevrey regularity of solutions in the $L^p$ framework, resorting only to Propositions \ref{prodlaws1}-\ref{prodlaws2} does not allow to get suitable bounds for the low frequency part of some nonlinear terms. The goal of this short subsection is to  establish more  estimates 
for the paraproduct, remainder operators in $L^2$ based Besov spaces, when 
the two functions under consideration belong to some $L^p$ type Besov space.
\begin{prop} \label{prodlaws3}
Assume that $2\leq p\leq\min (4, \frac{2d}{d-2})$ and $s\in \R$. There exists a constant $C>0$ such that
$$
\|\Gop T_{f}g\|_{\dot{B}^s_{2,1}} \leq C \|F\|_{\dot{B}^{\frac dp-1}_{p,1}} \|G\|_{\dot{B}^{s+1-\frac d2+\frac dp}_{p,1}}\with F\triangleq e^{\sqrt t\Lambda_1}f\andf G\triangleq e^{\sqrt t\Lambda_1}g.
$$
\end{prop}
\begin{proof}
If $p>2$ then we define  $p^*$ by the relation $\frac 12=\frac 1p + \frac1{p^*}\cdotp$ 
Then applying inequality \eqref{R-E30} with the exponents 
$(s,\sigma, p,p_1,p_2,r,r_1,r_2)=(s+1-\frac{d}{p^*},1-\frac{d}{p^*}, 2,p^*, p,1,1,\infty)$
which is possible since  $p^* \geq p$ (that is $2\leq p\leq 4$)
 and $-\sigma\triangleq \frac{d}{p^*}-1 \leq 0$ (or, equivalently, $p\leq \frac{2d}{d-2}$), we get
$$
\|\Gop T_{f}g\|_{\dot{B}^s_{2,1}} \leq C \|F\|_{\dot{B}^{\frac{d}{p^*}-1}_{p^*,1}} \|G\|_{\dot{B}^{s+1-\frac{d}{p^*}}_{p,\infty}}.
$$
Then using the embedding $\dot B^{\frac dp}_{p,1}\hookrightarrow \dot B^{\frac d{p^*}}_{p^*,1}$
(note that $p^* \geq p$) and $\dot{B}^{s+1-\frac d2+\frac dp}_{p,1}\hookrightarrow  \dot B^{\frac d{p^*}}_{p^*,1}$
gives the desired inequality.
\smallbreak
The endpoint case $p=2$ stems from  \eqref{R-E30b}  with the exponents 
$(s,\sigma,p,q,r,r_1,r_2)=(s+1,1,2,2,1,1,\infty)$.
\end{proof}

As a consequence of Proposition \ref{prodlaws1} and of the embedding $\dot{B}^{\sigma+d(\frac2p -\frac 12)}_{p/2,1}\hookrightarrow \dot{B}^{\sigma}_{2,1}$ for any $2\leq p \leq4$ and $\sigma\in\R,$ we readily get:
\begin{prop}\label{prodlaws4}
Let $d\geq2$ and $2\leq p \leq4$. If $s_1+s_2>d(\frac 12-\frac 2p)$ then there exists a constant $C>0$ such that
\begin{equation}\label{R-E82}
\|\Gop R(f,g)\|_{\dot{B}^{s_1+s_2}_{2,1}} \leq C \|F\|_{\dot{B}^{s_1+d\bigl(\frac 2p-\frac 12\bigr)}_{p,1}} \|G\|_{\dot{B}^{s_2}_{p,1}}.\end{equation}
\end{prop}

\begin{prop}  \label{proplaws5}
 \sl{Assume that $2\leq p\leq\min (4, \frac{2d}{d-2})$ and $p<2d.$  There exists a constant $C>0$ such that:
 \begin{equation}
  \begin{cases}
   \|\Goptau (fg)\|_{\dot{B}_{2,1}^{\frac d2}}^\ell \lesssim \|F\|_{\dot{B}_{p,1}^{\frac dp-1}} \|G\|_{\dot{B}_{p,1}^{\frac dp+1}} +\|F\|_{\dot{B}_{p,1}^{\frac dp+1}} \|G\|_{\dot{B}_{p,1}^{\frac dp-1}},\\
   \|\Goptau (fg)\|_{\dot{B}_{2,1}^{\frac d2-1}}^\ell \lesssim \|F\|_{\dot{B}_{p,1}^{\frac dp-1}} \|G\|_{\dot{B}_{p,1}^{\frac dp}} +\|F\|_{\dot{B}_{p,1}^{\frac dp}} \|G\|_{\dot{B}_{p,1}^{\frac dp-1}},\\
   \|\Goptau (fg)\|_{\dot{B}_{2,1}^{\frac d2-1}}^\ell \lesssim \|F\|_{\dot{B}_{p,1}^{\frac dp-1} \cap \dot{B}_{p,1}^{\frac dp}} \|G\|_{\dot{B}_{p,1}^{\frac dp-1}}.
  \end{cases}
  \label{prodlaws5}
  \end{equation}}\end{prop}
\begin{proof}
{}From Bony's decomposition, we have 
$$
\|\Goptau (fg)\|^{\ell}_{\dot{B}^{\frac d2}_{2,1}} = \|\Goptau (T_f g +T_g f +R(f,g))\|^{\ell}_{\dot{B}^{\frac d2}_{2,1}}.
$$
Thanks to Propositions \ref{prodlaws3} and \ref{prodlaws4} (with $(s,s_1,s_2)=(\frac d2, \frac d2 -\frac dp -1, \frac dp+1)$), we get that:
$$
\begin{cases}
 \|\Goptau (T_{f}g)\|_{\dot{B}^{\frac d2}_{2,1}}^{\ell} + \|\Goptau (R(f,g))\|_{\dot{B}^{\frac d2}_{2,1}}^{\ell} \lesssim \|{F}\|_{\dot{B}^{\frac dp-1}_{p,1}} \|{G}\|_{\dot{B}^{\frac dp+1}_{p,1}},\\
\|\Goptau (T_{g}f)\|_{\dot{B}^{\frac d2}_{2,1}}^{\ell}\lesssim \|{G}\|_{\dot{B}^{\frac dp-1}_{p,1}} \|{F}\|_{\dot{B}^{\frac dp+1}_{p,1}}.\\
\end{cases}
$$
The second estimate is proved the same way but with  $(s,s_1,s_2)=(\frac d2-1, \frac d2 -\frac dp -1, \frac dp)$. For the last estimate, we write that, taking advantage of the low frequency cut-off,
$$
\|\Goptau (fg)\|^{\ell}_{\dot{B}^{\frac d2-1}_{2,1}} \lesssim \|\Goptau (T_f g +T_g f)\|^{\ell}_{\dot{B}^{\frac d2-2}_{2,1}} +\|\Goptau R(f,g))\|^{\ell}_{\dot{B}^{\frac d2-1}_{2,1}}.
$$
The last term may be bounded as before, and for the first two terms, we apply 
Proposition \ref{prodlaws3} with  $s=\frac d2-2$.
\end{proof}


\subsection{A priori estimates for Gevrey regularity}

That paragraph is devoted to proving estimates  for Gevrey regularity in the $L^p$ Besov framework.
This will be based on the following lemma.
\begin{lem} \label{lem4.3}
If $(a,u)$ satisfies \eqref{R-E3}, then the following a priori estimate holds true:
\begin{eqnarray}\label{R-E85}
\|(a,u)\|_{Y_{p}}\leq C\bigl(\cX_{p,0}+\|(a,u)\|_{Y_{p}}^2+\|(a,u)\|_{Y_{p}}^3\bigr)\cdotp
\end{eqnarray}
\end{lem}
\begin{proof}
Summing up inequality \eqref{estimloc} for $j\leq k_0,$ we get for all $t\geq0,$
\begin{multline}\label{R-E86}
\|({A},{U})\|^{\ell}_{\tilde{L}^{\infty}_{t}(\dot{B}^{\frac{d}{2}-1}_{2,1})}
+\|({A},{U})\|^{\ell}_{{L}^{1}_{t}(\dot{B}^{\frac{d}{2}+1}_{2,1})}\\ \lesssim \|( a_0,u_0)\|^{\ell}_{\dot{B}^{\frac{d}{2}-1}_{2,1}}+\|F\|^{\ell}_{L^{1}_{t}(\dot{B}^{\frac{d}{2}-1}_{2,1})}
+\|G\|^{\ell}_{L^{1}_{t}(\dot{B}^{\frac{d}{2}-1}_{2,1})}.
\end{multline}

Regarding the high frequency estimates, we plan to repeat the computations of the previous section 
after introducing  $\Gop$ everywhere. 
Now,  using again the auxiliary functions  
$$v\triangleq \cQ u+(-\Delta)^{-1}\nabla a\andf w\triangleq v+\alpha\nabla a\with\alpha=\frac12\bigl(1+\sqrt{1-4\kab}\bigr),$$ and setting $\wt\alpha\triangleq1-\alpha$ and $\tilde{g}\triangleq\cQ g+(-\Delta)^{-1}\nabla f+v-(-\Delta)^{-1}\nabla a,$ we discover that 
$$w(t)=e^{\wt \alpha t\Delta}w_0+\int^t_{0}e^{\wt\alpha(t-\tau)\Delta}(-\alpha\nabla a+\alpha\nabla f+\tilde{g})(\tau)\,d\tau.$$
Hence ${W(t)}\triangleq \Gop w(t)$ fulfills (with obvious notation):
 $$
W(t)=e^{\sqrt{t}\Lambda_1+\wt\alpha t\Delta}w_0+\int^t_{0}e^{[(\sqrt{t}-\sqrt{\tau})\Lambda_1+\wt\alpha(t-\tau)\Delta]} (-\alpha\nabla A+\alpha\nabla F+\tilde{G})(\tau)\,d\tau. $$
It follows from Lemmas \ref{lem3.2}-\ref{lem3.3} that for the same threshold $k_0$ as in \eqref{R-E73} and \eqref{R-E74}, we have
$$\begin{aligned}
\|{W}\|^{h}_{\tilde{L}^{\infty}_{T}(\dot{B}^{\frac{d}{p}-1}_{p,1})}\!+\!\|{W}\|^{h}_{\tilde{L}^{1}_{T}(\dot{B}^{\frac{d}{p}+1}_{p,1})}&\lesssim \|w_0\|^{h}_{\dot{B}^{\frac dp-1}_{p,1}}+\|{A}\|^{h}_{L^{1}_{T}(\dot{B}^{\frac{d}{p}}_{p,1})}
\!+\!\|F\|^{h}_{L^{1}_{T}(\dot{B}^{\frac{d}{p}}_{p,1})}
\!+\!\|\wt G\|^{h}_{L^{1}_{T}(\dot{B}^{\frac{d}{p}-1}_{p,1})}\\
&\lesssim \|w_0\|^{h}_{\dot{B}^{\frac dp-1}_{p,1}}+2^{-2k_{0}}\|{A}\|^{h}_{L^{1}_{T}(\dot{B}^{\frac{d}{p}+2}_{p,1})}+2^{-2k_{0}}\|{V}\|^{h}_{L^{1}_{T}(\dot{B}^{\frac{d}{p}+1}_{p,1})}\\
&\hspace{4cm}+\|F\|^{h}_{L^{1}_{T}(\dot{B}^{\frac{d}{p}}_{p,1})}
+\| G\|^{h}_{L^{1}_{T}(\dot{B}^{\frac{d}{p}-1}_{p,1})}.
\end{aligned}$$
Then one can revert to $v$ as in \eqref{eq:Da}, applying $\Gop$ to:
$$\partial_{t}v-\frac{\kab}{1-\alpha} \Delta v=\frac{\kab}{\alpha} \Delta w+\tilde{g}.$$
Denoting   ${V}\triangleq \Gop v$ and   following  the procedure leading to \eqref{R-E75}, one gets
$$\begin{aligned}
\|{V}\|^{h}_{\tilde{L}^{\infty}_{T}(\dot{B}^{\frac{d}{p}-1}_{p,1})}+\|{V}\|^{h}_{\tilde{L}^{1}_{T}(\dot{B}^{\frac{d}{p}+1}_{p,1})}&\lesssim \|W\|^h_{L^{1}_{T}(\dot{B}^{\frac{d}{p}+1}_{p,1})}
+\|\wt G\|^h_{L^{1}_{T}(\dot{B}^{\frac{d}{p}-1}_{p,1})}\\
&\lesssim \|v_0\|^{h}_{\dot{B}^{\frac dp-1}_{p,1}}+\|{W}\|^{h}_{\tilde{L}^{1}_{T}(\dot{B}^{\frac{d}{p}+1}_{p,1})}
+2^{-2k_{0}}\|{V}\|^{h}_{L^{1}_{T}(\dot{B}^{\frac{d}{p}+1}_{p,1})}
\\&\ +2^{-4k_{0}}\|{A}\|^{h}_{L^{1}_{T}(\dot{B}^{\frac{d}{p}+2}_{p,1})}+\|F\|^{h}_{L^{1}_{T}(\dot{B}^{\frac{d}{p}-2}_{p,1})}+\|G\|^{h}_{L^{1}_{T}(\dot{B}^{\frac{d}{p}-1}_{p,1})}.
\end{aligned}
$$
For the incompressible part of velocity, applying $\Gop$ to \eqref{eq:incompr} yields
\begin{equation}\label{R-E96}
\|\mathcal{P}{U}\|^{h}_{\tilde{L}^{\infty}_{T}(\dot{B}^{\frac{d}{p}-1}_{p,1})}+\|\mathcal{P}{U}\|^{h}_{\tilde{L}^{1}_{T}(\dot{B}^{\frac{d}{p}+1}_{p,1})}
\lesssim \|\mathcal{P}u_0\|_{\dot{B}^{\frac{d}{p}-1}_{p,1}}+\|G\|^{h}_{L^{1}_{T}(\dot{B}^{\frac{d}{p}-1}_{p,1})}.
\end{equation}
Therefore, taking  the same large enough $k_0$ as in the previous section, 
and using \eqref{eq:Da}, we  deduce that
\begin{multline}\label{R-E97}
\|(\nabla {A},{U})\|^{h}_{\tilde{L}^{\infty}_{t}(\dot{B}^{\frac{d}{p}-1}_{p,1})}+\|(\nabla {A},{U})\|^{h}_{\tilde{L}^{1}_{t}(\dot{B}^{\frac{d}{p}+1}_{p,1})} \\ \lesssim \|(\nabla a_0,u_0)\|^{h}_{\dot{B}^{\frac{d}{p}-1}_{p,1}}+\|F\|^{h}_{L^{1}_{t}(\dot{B}^{\frac{d}{p}}_{p,1})}
+\|G\|^{h}_{L^{1}_{t}(\dot{B}^{\frac{d}{p}-1}_{p,1})}.
\end{multline}
Putting together with  \eqref{R-E86},  we end up with 
\begin{multline}\label{R-E98}
\|({A},{U})\|_{X_p(t)}\lesssim X_{p,0}+\|F\|^{\ell}_{L^{1}_{t}(\dot{B}^{\frac{d}{2}}_{2,1})}
+\|G\|^{\ell}_{L^{1}_{t}(\dot{B}^{\frac{d}{2}-1}_{2,1})}\\+\|F\|^{h}_{L^{1}_{t}(\dot{B}^{\frac{d}{p}}_{p,1})}
+\|G\|^{h}_{L^{1}_{t}(\dot{B}^{\frac{d}{p}-1}_{p,1})}.
\end{multline}
All that remains to do is to bound $F$ and $G,$ which will be strongly based on Proposition \ref{proplaws5}
as regards the low frequencies. 
\smallbreak
Let us start with $F.$ Then, thanks to \eqref{prodlaws5}$_1$ and  Besov injections (as $p\geq 2$), we get 
$$\begin{aligned}
\|F\|^{\ell}_{L^1_{t}(\dot{B}^{\frac d2-1}_{2,1})} &\lesssim \|{A}\|_{L^\infty_{t}(\dot{B}^{\frac dp-1}_{p,1})}
\|{U}\|_{L^1_{t}(\dot{B}^{\frac dp+1}_{p,1})} +\|{U}\|_{L^\infty_{t}(\dot{B}^{\frac dp-1}_{p,1})}
\|{A}\|_{L^1_{t}(\dot{B}^{\frac dp+1}_{p,1})}\\
&\lesssim  \Big(\|{A}\|^{\ell}_{L^\infty_{t}(\dot{B}^{\frac d2-1}_{2,1})}+ \|{A}\|^{h}_{L^\infty_{t}(\dot{B}^{\frac dp}_{p,1})}\Big) \Big(\|{U}\|^{\ell}_{L^1_{t}(\dot{B}^{\frac d2+1}_{2,1})} +\|{U}\|^{h}_{L^1_{t}(\dot{B}^{\frac dp+1}_{p,1})}\Big)\\
&\qquad\qquad+\Big(\|{A}\|^{\ell}_{L^1_{t}(\dot{B}^{\frac d2+1}_{2,1})}\!+\!\|{A}\|^h_{L^1_{t}(\dot{B}^{\frac dp+2}_{p,1})}\Big) \Big(\|{U}\|^{\ell}_{L^\infty_{t}(\dot{B}^{\frac d2-1}_{2,1})} \!+\!\|{U}\|^h_{L^\infty_{t}(\dot{B}^{\frac dp-1}_{p,1})}\Big) 
\\&\lesssim \|(a,u)\|_{Y_p}^2.
\end{aligned}$$
Next, we bound the norm $\|G\|^{\ell}_{L^1_{t}(\dot{B}^{\frac d2-1}_{2,1})}$. Using \eqref{prodlaws5}$_2$ we obtain
$$\begin{aligned}
 \|G_1\|^{\ell}_{L^1_{t}(\dot{B}^{\frac d2-1}_{2,1})} &=\|\Goptau (u\cdot \nabla u)\|^{\ell}_{L^1_{t}(\dot{B}^{\frac d2-1}_{2,1})}\\
 &\lesssim \|{U}\|_{L^\infty_{t}(\dot{B}^{\frac dp-1}_{p,1})}\|{U}\|_{L^1_{t}(\dot{B}^{\frac dp+1}_{p,1})} +\|{U}\|_{L^2_{t}(\dot{B}^{\frac dp}_{p,1})}^2 \lesssim \|(a,u)\|_{Y_p}^2.
\end{aligned}$$
Let us now turn to $G_3 =-\Goptau(I(a)\cAb u).$ Thanks to \eqref{prodlaws5}$_3$ and Proposition \ref{l:compo}:
$$\begin{aligned}
 \|G_3\|^{\ell}_{L^1_{t}(\dot{B}^{\frac d2-1}_{2,1})} &\lesssim \int_0^t \|\Goptau I(a)\|_{\dot{B}_{p,1}^{\frac dp-1} \cap \dot{B}_{p,1}^{\frac dp}} \|U\|_{\dot{B}_{p,1}^{\frac dp+1}} d\tau \\
&\lesssim \|{A}\|_{L^\infty_{t}(\dot{B}^{\frac dp-1}_{p,1} \cap \dot{B}^{\frac dp}_{p,1})}\|{U}\|_{L^1_{t}(\dot{B}^{\frac dp+1}_{p,1})} \lesssim \|(a,u)\|_{Y_p}^2.
\end{aligned}$$
Similarly, we estimate $G_4\triangleq \Goptau(J(a)\nabla a)$ using \eqref{prodlaws5}$_2$ and Proposition \ref{l:compo}:
$$\begin{aligned}\|G_4\|^{\ell}_{L^1_{t}(\dot{B}^{\frac d2-1}_{2,1})}
  &\lesssim \int_0^t\biggl(\|\Goptau J(a)\|_{\dot{B}_{p,1}^{\frac dp-1}} \|\nabla A\|_{\dot{B}_{p,1}^{\frac dp}} +\|\Goptau J(a)\|_{\dot{B}_{p,1}^{\frac dp}} \|\nabla A\|_{\dot{B}_{p,1}^{\frac dp-1}}\biggr) d\tau\\  
   &\lesssim \int_0^t\biggl(\|A\|_{\dot{B}_{p,1}^{\frac dp-1}} \|\nabla A\|_{\dot{B}_{p,1}^{\frac dp}} +\|A\|_{\dot{B}_{p,1}^{\frac dp}} \|\nabla A\|_{\dot{B}_{p,1}^{\frac dp-1}}\biggr) d\tau\\  
  &\lesssim \|(a,u)\|_{Y_p}^2.\end{aligned}$$
In order to bound the term corresponding to $g_3,$
it suffices to consider $\wt G_3\triangleq \Goptau(1-I(a))\nabla (\tmu(a)\nabla u),$ the other 
term being similar. Now,  we have:
$$ \|G_2\|^{\ell}_{L^1_{t}(\dot{B}^{\frac d2-1}_{2,1})}\leq
 \underbrace{\|\Goptau \big(\tmu(a)\nabla u\big)\|^{\ell}_{L^1_{t}(\dot{B}^{\frac d2}_{2,1})}}_{I} 
 +\underbrace{\|\Goptau \Big(I(a)\nabla \big(\tmu(a)\nabla u\big) \Big)\|^{\ell}_{L^1_{t}(\dot{B}^{\frac d2-1}_{2,1})}}_{II}.$$
 The first term may be bounded (taking once again advantage of the low frequencies cut-off) according to \eqref{prodlaws5}$_2$ and Proposition \ref{l:compo} as follows:
$$ \begin{aligned}
  I &\lesssim \|\Goptau \big(\tmu(a)\nabla u\big)\|^{\ell}_{L^1_{t}(\dot{B}^{\frac d2-1}_{2,1})}\\
  &\lesssim \int_0^t \Big( \|\Goptau \big( \tmu (a)\big)\|_{\dot{B}^{\frac dp-1}_{p,1}} \|U\|_{\dot{B}^{\frac dp+1}_{p,1}} +\|\Goptau \big( \tmu (a)\big)\|_{\dot{B}^{\frac dp}_{p,1}} \|U\|_{\dot{B}^{\frac dp}_{p,1}} \Big) d\tau.
 \end{aligned}$$
 The second term is bounded using \eqref{prodlaws5}$_3$ and Propositions \ref{l:compo} and \ref{prop3.4}:
 $$\begin{aligned}
 II &\lesssim \int_0^t \|\Goptau I(a)\|_{\dot{B}^{\frac dp-1}_{p,1} \cap \dot{B}^{\frac dp}_{p,1}} \|\Goptau \nabla \big(\tmu(a)\nabla u\big)\|_{\dot{B}^{\frac dp-1}_{p,1}} d\tau\\
 &\lesssim \int_0^t \|A\|_{\dot{B}^{\frac dp-1}_{p,1} \cap \dot{B}^{\frac dp}_{p,1}} \|\Goptau \tmu(a)\|_{\dot{B}_{p,1}^\frac dp} \|U\|_{\dot{B}^{\frac dp+1}_{p,1}}\,d\tau.
\end{aligned}$$
We finally obtain that
$$\|G_2\|^{\ell}_{L^1_{t}(\dot{B}^{\frac d2-1}_{2,1})} \lesssim (1+\|(a,u)\|_{Y_p}) \|(a,u)\|_{Y_p}^2.$$
To bound the capillary terms, we use  \eqref{prodlaws5}$_2,$ writing that
$$\begin{aligned}
 \|\Goptau \nabla(\tilde{\kappa}(a)\Delta a)\|^{\ell}_{L^1_{t}(\dot{B}^{\frac d2-1}_{2,1})}
 & \lesssim \|\Goptau (\tilde{\kappa}(a)\Delta a)\|^{\ell}_{L^1_{t}(\dot{B}^{\frac d2-1}_{2,1})}\\
 &\lesssim  \|\Goptau \tilde{\kappa}(a)\|_{L^\infty_{t}(\dot{B}^{\frac dp-1}_{p,1})}\|\Delta {A}\|_{L^1_{t}(\dot{B}^{\frac dp}_{p,1})} \\&\hspace{3cm}+\|\Delta {A}\|_{L^1_{t}(\dot{B}^{\frac dp-1}_{p,1})}\|\Goptau \tilde{\kappa}(a)\|_{L^\infty_{t}(\dot{B}^{\frac dp}_{p,1})}\\
 &\lesssim  \|{A}\|_{L^\infty_{t}(\dot{B}^{\frac dp-1}_{p,1})}\|{A}\|_{L^1_{t}(\dot{B}^{\frac dp+2}_{p,1})}+\| {A}\|_{L^1_{t}(\dot{B}^{\frac dp+1}_{p,1})}\|{A}\|_{L^\infty_{t}(\dot{B}^{\frac dp}_{p,1})}.
\end{aligned}$$
As we just have to bound the low frequencies, one gets  thanks to \eqref{prodlaws5}$_2,$
$$\begin{aligned}
 \Big\|\Goptau \nabla \Big(\frac{1}{2}\nabla \tka(a)\cdot \nabla a\Big)\Big\|^{\ell}_{L^1_{t}(\dot{B}^{\frac d2-1}_{2,1})} &\lesssim\|\Goptau (\nabla \tka (a)\cdot\nabla a)\|^{\ell}_{L^1_{t}(\dot{B}^{\frac d2-1}_{2,1})}\\
&\lesssim \int_0^t \Bigl(\|\Goptau (\nabla \tka (a))\|_{\dot{B}_{p,1}^{\frac dp-1}} \|\nabla A\|_{\dot{B}_{p,1}^{\frac dp}}
\\&\hspace{1.8cm} +\|\Goptau (\nabla \tka (a))\|_{\dot{B}_{p,1}^{\frac dp}} \|\nabla A\|_{\dot{B}_{p,1}^{\frac dp-1}}\Bigr)d\tau.
\end{aligned}$$
Thanks to Proposition \ref{l:compo} we see  that the first term is bounded by:
$$\begin{aligned}
 \|\Goptau \tka (a)\|_{L^\infty_{t}(\dot{B}^{\frac dp}_{p,1})}&\|A\|_{L^1_{t}(\dot{B}^{\frac dp+1}_{p,1})} \lesssim \|{A}\|_{L^\infty_{t}(\dot{B}^{\frac dp}_{p,1})}\|{A}\|_{L^1_{t}(\dot{B}^{\frac dp+1}_{p,1})}\\
&\lesssim  \Big(\|{A}\|^{\ell}_{L^\infty_{t}(\dot{B}^{\frac d2-1}_{2,1})}\!+\! \|{A}\|^{h}_{L^\infty_{t}(\dot{B}^{\frac dp}_{p,1})}\Big) \Big(\|{A}\|^{\ell}_{L^1_{t}(\dot{B}^{\frac d2+1}_{2,1})} \!+\!\|{A}\|^{h}_{L^1_{t}(\dot{B}^{\frac dp+2}_{p,1})}\Big).\\
\end{aligned}$$
We have to be careful for the last term as $\frac dp+1$ is not in the range of Proposition \ref{l:compo}. 
However, we have  $\nabla \tka (a)=\tka' (a)\nabla a$ and thus, 
$$\begin{aligned}
 \int^{t}_{0}\|\Goptau (\tka' (a)&\nabla a)\|_{\dot{B}^{\frac dp}_{p,1}} \|{A}\|_{\dot{B}^{\frac dp}_{p,1}}d\tau\\
 &\lesssim
 \int^{t}_{0}\Big( \|\Goptau \big(\tka' (a)-\tka' (0)\big)\|_{\dot{B}^{\frac dp}_{p,1}} +|\tka' (0)| \Big)\|\nabla A\|_{\dot{B}^{\frac dp}_{p,1}} \|{A}\|_{\dot{B}^{\frac dp}_{p,1}} d\tau\\
  &\lesssim (1+\|{A}\|_{L^\infty_{t}(\dot{B}^{\frac dp}_{p,1})}) \|{A}\|_{L^\infty_{t}(\dot{B}^{\frac dp}_{p,1})} \|{A}\|_{L^1_{t}(\dot{B}^{\frac dp+1}_{p,1})},
\end{aligned}$$
which enables us to obtain:
$$\|\Goptau g_5(\tau)\|^{\ell}_{L^1_{t}(\dot{B}^{\frac d2-1}_{2,1})} \lesssim (1+\|(a,u)\|_{Y_p}) \|(a,u)\|_{Y_p}^2.$$
To complete the proof, we need to bound the high frequencies of $F$ and $G.$ This 
turns out to be rather straightforward, as we only need Proposition \ref{prop3.6} and 
Lemma \ref{l:compo}. More precisely, we get
$$\begin{aligned}
\|F\|^h_{L^1_{t}(\dot{B}^{\frac dp}_{p,1})} &\lesssim \|\Goptau (a\div u+
u\cdot\nabla a)\|^h_{L^1_{t}(\dot{B}^{\frac dp}_{p,1})}
 \\ &\lesssim
\|{A}\|_{L^\infty_{t}(\dot{B}^{\frac dp}_{p,1})}\|{\div U}\|_{L^1_{t}(\dot{B}^{\frac dp}_{p,1})}+\|{U}\|_{L^2_{t}(\dot{B}^{\frac dp}_{p,1})}\|{\nabla A}\|_{L^2_{t}(\dot{B}^{\frac dp}_{p,1})}
\end{aligned}$$
and 
$$\displaylines{
\|G\|^h_{L^1_{t}(\dot{B}^{\frac dp-1}_{p,1})}\lesssim \|{U}\|_{L^\infty_{t}(\dot{B}^{\frac dp-1}_{p,1})}\|{U}\|_{L^1_{t}(\dot{B}^{\frac dp+1}_{p,1})}+\|{A}\|^2_{L^2_{t}(\dot{B}^{\frac dp}_{p,1})}
\hfill\cr\hfill+\|{A}\|_{L^\infty_{t}(\dot{B}^{\frac dp}_{p,1})}\|{U}\|_{L^1_{t}(\dot{B}^{\frac dp+1}_{p,1})}+\|{A}\|_{L^\infty_{t}(\dot{B}^{\frac dp}_{p,1})}\|{A}\|_{L^1_{t}(\dot{B}^{\frac dp+2}_{p,1})}\hfill\cr\hfill+(1+\|{A}\|_{L^\infty_{t}(\dot{B}^{\frac dp}_{p,1})})\|{A}\|^2_{L^2_{t}(\dot{B}^{\frac dp+1}_{p,1})}.}$$
Putting all the previous estimates  together ends the proof of Lemma \ref{lem4.3}.
\end{proof}
Finally, as in the previous section, using a suitable contracting mapping argument enables us to complete the proof of  Theorem \ref{thm4.1}.
The details are left to the reader. 
As for uniqueness, it stems from  \cite[Thm. 5]{DD}.


\subsection{The time-decay of solutions in Besov spaces}\setcounter{equation}{0}\label{sec:5}
The aim of this part is to  exhibit 
 the time-decay properties of the solutions that have been constructed in Theorems \ref{thm3.1} and \ref{thm4.1}.  Those properties will come up as a consequence of the following lemma. 
 \begin{lem} There exists a universal constant $c>0$  such that for all  $s\in\R,$ there exists  
 a constant $C_s$  such that for any 
tempered distribution $u,$ real number $\alpha>0$ and integer $j\in\Z,$ the following
inequality holds true:
\begin{equation}\label{eq:expdecay}\|\Lambda^se^{-\alpha \Lambda_1}\ddj u\|_{L^p}
\leq C_s2^{js} e^{-c\alpha 2^j} \|\ddj u\|_{L^p}.\end{equation}
 \end{lem}
 \begin{proof}
 The starting point is the fact that, by definition of operator $e^{-\alpha\Lambda_1},$ we have
 for all $v\in\cS'(\R^d),$ 
 $$ e^{-\alpha\Lambda_1} v= h_\alpha\star v\with
 h_\alpha=\cF^{-1}(e^{-\alpha |\cdot|_1})\cdotp
 $$
 Now, we notice that $h_\alpha$ is nonnegative, since
 $$\int_\R e^{-|\eta|}e^{ix\eta}\,d\eta=\frac{2}{1+x^2}$$
 and,  owing to the definition of $|\xi|_1,$ we have
 $$
 \cF^{-1}(e^{-\alpha |\cdot|_1})(x)=\frac1{(2\pi)^d}
 \prod_{j=1}^d\biggl(\int_\R e^{-\alpha |\xi_j|}e^{ix_j\xi_j}\,d\xi_j\biggr)\cdotp
$$
Therefore
$$
\|h_\alpha\|_{L^1}=\int_{\R^d} h_\alpha(x)\,dx = \cF( \cF^{-1}(e^{-\sqrt\alpha |\cdot|_1}))(0)=1.
$$
From this, we deduce by Young inequality that for all $\alpha\geq0,$
\begin{equation}\label{eq:decay1}
\|e^{-\alpha \Lambda_1}v\|_{L^p}\leq \|v\|_{L^p}.
\end{equation}
In order to get \eqref{eq:expdecay} for $s=0,$ one has to refine the argument. 
First, performing a suitable rescaling reduces the proof to the case $j=0.$
Then we introduce a family   $(\phi_k)_{1\leq k\leq d}$ of smooth functions on $\R^d$ such that
\begin{enumerate}
\item $\Supp\phi_k\subset\bigl\{\xi\in\R^d\;,\; \frac34\leq|\xi|\leq \frac83 \andf \frac3{4\sqrt{d}}\leq |\xi_k|\bigr\};$
\item $\sum_{k=1}^d \phi_k\equiv 1$ on $\Supp\varphi$, where $\varphi$ is the function
used in the definition of the Littlewood-Paley decomposition. 
\end{enumerate}
As we obviously have
$$
e^{-\alpha|\xi|_1} \cF(\dot\Delta_0 u)(\xi)=\sum_{k=1}^d (e^{-\alpha|\xi|_1} \phi_k(\xi))\: \cF(\dot\Delta_0 u)(\xi),
$$
one may write 
$$
e^{-\alpha \Lambda_1}\dot\Delta_0 u=\sum_{k=1}^d h_k\star \dot\Delta_0 u
\with h_k\triangleq \cF^{-1}(e^{-\alpha|\cdot|_1} \phi_k).$$
If we prove that  for some $c>0$ and $C>0$ independent of $\alpha,$  we have
\begin{equation}\label{eq:hk} 
\|h_k\|_{L^1}\leq C\biggl(\frac{1+\alpha}{\alpha}\biggr)^de^{-c\alpha},
\end{equation} then, combining with \eqref{eq:decay1} will complete the proof of the lemma for $s=0.$ 
 \medbreak
 Let us prove \eqref{eq:hk} for $k=1$ (the other cases being similar). 
  Then we introduce the notation $\xi=(\xi_1,\xi')$ and $x=(x_1,x').$ Since  $(\alpha^2+x_1^2) e^{ix\cdot\xi} = (\alpha^2-\d^2_{\xi_1\xi_1})(e^{ix\cdot\xi}),$ 
integrating by parts with respect to the variable $\xi_1$ in the integral defining  $h_1$  yields:
$$(\alpha^2+x_1^2) h_1(x)=\frac1{(2\pi)^d}\int_{\R^d} e^{ix\cdot \xi} e^{-\alpha|\xi'|_1}(\alpha^2-\d^2_{11})(\phi_1(\xi)e^{-\alpha|\xi_1|})\,d\xi.$$
Now, let us observe that 
$$
(e^{-\alpha |r|})'=-\alpha e^{-\alpha|r|} \sgn r\andf\alpha^2 e^{-\alpha |r|} - (e^{-\alpha |r|})'' = 2\alpha \delta_0.
$$
Therefore, 
$$
(\alpha^2-\d^2_{11})(\phi_1(\xi)e^{-\alpha|\xi_1|})=2\alpha\phi_1(0, \xi') \delta_{\xi_1=0} 
+e^{-\alpha|\xi_1|} \bigl(2\alpha \sgn(\xi_1)\d_1\phi_1(\xi)-\d^2_{11} \phi_1(\xi)\bigr),
$$
and thus (taking advantage of the fact that $\phi_1(0,\xi')=0$)
$$\displaylines{\quad(\alpha^2+x_1^2) h_1(x)=\frac1{(2\pi)^d} \int_{\R^d} e^{ix\cdot \xi} e^{-\alpha|\xi|_1}(2\alpha \sgn(\xi_1)\d_1\phi_1(\xi)-\d^2_{11}\phi_1(\xi))\,d\xi \cdotp\quad}
$$
Multiplying by $(\alpha^2+x_2^2)$, the same arguments lead to (denoting $\xi_2'=(\xi_1,0, \xi_2,...,\xi_d)$ and $\phi_1^2(\xi)= 2\alpha \sgn(\xi_1)\d_1\phi_1(\xi)-\d^2_{11} \phi_1(\xi)$)
$$\displaylines{\quad(\alpha^2+x_1^2)(\alpha^2+x_2^2) h_1(x) =\frac1{(2\pi)^d}\biggl(2\alpha\int_{\R^{d-1}}e^{ix_2'\cdot\xi_2'}e^{-\alpha|\xi_2'|_1}\phi_1^2(\xi_2')\,d\xi_2'\hfill\cr\hfill+
\int_{\R^d} e^{ix\cdot \xi} e^{-\alpha|\xi|_1}(2\alpha \sgn(\xi_2)\d_2\phi_1^2(\xi)-\d^2_{22})\phi_1^2(\xi)\,d\xi\biggr)\cdotp\quad}
$$
Multiplying the above equality by $ (\alpha^2+x_3^2)\dotsm(\alpha^2+x_{d}^2),$ repeating 
the above computation, and using the fact that, 
$$\forall \xi\in \Supp\phi_1,\;
e^{-\alpha|\xi|_1}\leq e^{-\alpha|\xi_1|}\leq e^{-\frac{3\alpha}{4\sqrt d}},$$
we end up with  
$$\prod_{\ell=1}^d (\alpha^2+x_\ell^2) h_1(x)\leq C (\alpha+1)^de^{-\frac{3\alpha}{4\sqrt d}},$$ 
which implies \eqref{eq:hk}, and thus the lemma for $s=0.$
\bigbreak
Proving the general case $s\geq0$ follows from the case $s=0$: indeed, Inequality  
 \eqref{R-E14b} ensures that
$$\|\Lambda^{s}e^{-\alpha \Lambda_1}\ddj u\|_{L^p}\leq C_s 2^{js}  \|e^{-\alpha \Lambda_1}\ddj u\|_{L^p},$$
and bounding the right-hand side according to \eqref{eq:expdecay} thus yields the desired inequality.
 \end{proof}
One can now state our main decay estimates.
\begin{thm}\label{thm5.1}
Let $(\varrho,u)$ be the solution constructed in Theorem \ref{thm4.1}. Then for any $s\in[0,\infty[,$ 
there exists a constant $C_s$  such that for all $t>0,$ it holds that
$$\begin{aligned}
&\|\varrho(t)-\bar{\varrho}\|_{\dot{B}^{\frac {d}{2}-1+s}_{2,1}}^\ell\leq C_s  X_{p,0}t^{-\frac s2},\qquad
&&\|u(t)\|_{\dot{B}^{\frac {d}{2}-1+s}_{2,1}}^\ell\leq C_s X_{p,0}t^{-\frac s2},\\
&\|\varrho(t)-\bar{\varrho}\|_{\dot{B}^{\frac {d}{p}+s}_{p,1}}^h\leq C_s X_{p,0}t^{-\frac s2}e^{-c\sqrt t},\qquad
&&\|u(t)\|_{\dot{B}^{\frac {d}{p}-1+s}_{p,1}}^h\leq C_s X_{p,0}t^{-\frac s2}e^{-c\sqrt t}\cdotp
\end{aligned}$$
\end{thm}
\begin{proof}
Recall that the solution constructed in Theorem \ref{thm4.1} fulfills
$$\|(\varrho-\bar{\varrho}, u)\|_{Y_p}\leq  CX_{p,0}.$$
Now,  Inequality \eqref{R-E14b} implies that
$$
\|u(t)\|_{\dot{B}^{\frac {d}{2}-1+s}_{2,1}}^\ell\leq C_s\|\Lambda^su(t)\|_{\dot{B}^{\frac {d}{2}-1}_{2,1}}^\ell.
$$
Then we write, denoting $U= e^{\sqrt{c_0t}\Lambda_1}u$ and using the previous lemma, that 
$$\begin{aligned}
t^{\frac s2}\|\Lambda^su\|^\ell_{\dot B^{\frac d2-1}_{2,1}}
&=\sum_{j\leq k_0} t^{\frac s2}2^{j(\frac d2-1)}\|\Lambda^se^{-\sqrt{c_0t}\Lambda_1}\ddj U(t)\|_{L^2}\\
&\leq C_s \sum_{j\leq k_0} (\sqrt t2^{j})^s e^{-c\sqrt{c_0t}2^j} 2^{j(\frac d2-1)}\|\ddj U(t)\|_{L^2}\\
&\leq C_s \|U(t)\|_{\dot B^{\frac d2-1}_{2,1}}^\ell\\
&\leq C_s X_{p,0}.
\end{aligned}
$$
Similarly, we have
$$\begin{aligned}
t^{\frac s2}\|u(t)\|_{\dot{B}^{\frac {d}{p}-1+s}_{p,1}}^h&\leq 
C_s\|\Lambda^su(t)\|_{\dot{B}^{\frac {d}{p}-1}_{p,1}}^h \\
&\leq C_s \sum_{j\geq k_0} 2^{j(\frac dp-1)} t^{\frac s2}\|e^{-\sqrt{c_0t}\Lambda_1}\Lambda^s\ddj U(t)\|_{L^p}\\
&\leq C_s\sum_{j\geq k_0}  e^{-\frac c2\sqrt{c_0t}2^j}2^{j(\frac dp-1)} (\sqrt t 2^j)^s e^{-\frac c2\sqrt{c_0t}2^j}
\|\ddj U(t)\|_{L^p}\\
&\leq C_s  e^{-\frac c2\sqrt{c_0t}2^{k_0}} \|U(t)\|_{\dot B^{\frac dp-1}_{p,1}}^h\\
&\leq C_s e^{-\frac c2\sqrt{c_0t}2^{k_0}} X_{p,0}.\end{aligned}
$$
Proving the inequalities for $\varrho$ is totally similar. 
\end{proof}
\begin{rem}\label{rem5.1}
The decay estimate pointed out  in Theorem  \ref{thm5.1}  is much better 
than that of the usual compressible Navier-Stokes (see for example \cite{DX}).
This reflects the parabolicity of the compressible Navier-Stokes-Korteweg system.
\end{rem}



\appendix
\section{Littlewood-Paley Decomposition and Besov Spaces}\setcounter{equation}{0}\label{sec:2}

Here we recall a few basic results concerning the Littlewood-Paley decomposition
and Besov spaces. More details may be found  in e.g. \cite[Chap. 2]{BCD}. 

To build the Littlewood-Paley decomposition, one
need a smooth radial function $\chi$ supported in the ball $\mathcal{B}(0,\frac{4}{3})$ and  with value $1$ on $\mathcal{B}(0,\frac{3}{4}).$ Let $\varphi(\xi)\triangleq\chi(\xi/2)-\chi(\xi).$
Then,  $\varphi$ is compactly supported in the annulus $\{\xi\in \mathbb{R}^d, \frac{3}{4}\leq|\xi|\leq\frac{8}{3}\}$ 
and fulfills
$$
\sum_{j\in\Z}\varphi(2^{-j}\cdot)=1\ \hbox{ in }\ \R^d\setminus\{0\}\cdotp
$$
Define the dyadic blocks $(\dot{\Delta}_j)_{j\in \mathbb{Z}}$ by $\dot{\Delta}_{j}=\varphi(2^{-j}D)$ (that is, 
$\widehat{\dot{\Delta}_{j}f}:=\varphi(2^{-j}\xi)\hat{f}(\xi)$ for all tempered distribution $f$). 
 The (formal) homogeneous Littlewood-Paley decomposition of $f$ reads
$$f=\sum_{j\in\Z}\ddj f.$$
That equality holds true in the set $\cS'$ of  tempered distributions 
whenever $f$ belongs to 
 $$\mathcal{S}'_{h}\triangleq \bigl\{f\in\cS',\ | \ \lim_{j\rightarrow-\infty}\|\dot{S}_{j}f\|_{L^\infty}=0\bigr\},$$
  where $\dot{S}_{j}$ stands for the low frequency cut-off defined by $\dot{S}_{j}=\chi(2^{-j}D)$.
  
\begin{defn}\label{defn2.1}
 For $\sigma\in\R$ and
$1\leq p,r\leq\infty,$ we set
$$
\|f\|_{\dot B^\sigma_{p,r}}=\Bigl\|2^{j\sigma}\|\ddj  f\|_{L^p(\R^d)}\Bigr\|_{\ell^r(\Z)}.
$$
We then define the homogeneous Besov space $\dot B^\sigma_{p,r}$ to be the
subset of  distributions $f\in {\cS}'_h$ such  that
$\|f\|_{\dot B^\sigma_{p,r}}<\infty.$
\end{defn}
 Homogeneous Besov spaces on $\R^d$ possess the following
 scaling invariance  for any $\sigma\in\R$ and $(p,r)\in[1,+\infty]^2$:
\begin{equation}\label{R-E7}
C^{-1}\lambda^{\sigma-\frac dp} \|f\|_{\dot B^\sigma_{p,r}}\leq
\|f(\lambda\cdot)\|_{\dot B^\sigma_{p,r}}\leq C\lambda^{\sigma-\frac dp}  \|f\|_{\dot B^\sigma_{p,r}},\qquad\lambda>0,
\end{equation}
where the constant $C$ depends only on $\sigma,$ $p$ and on the dimension $d.$
\medbreak
The following properties have been used repeatedly in the paper:
\begin{itemize}
\item The space $\dot B^s_{p,r}$ is complete whenever
$s<d/p,$ or $s\leq d/p$ and $r=1$.
  \item For any $p\in[1,\infty],$ we have the  continuous embedding $\dot B^0_{p,1}\hookrightarrow L^p\hookrightarrow \dot B^0_{p,\infty}.$
\item If  $\sigma\in\R,$ $1\leq p_1\leq p_2\leq\infty$ and $1\leq r_1\leq r_2\leq\infty,$
  then $\dot B^{\sigma}_{p_1,r_1}\hookrightarrow
  \dot B^{\sigma-d(\frac1{p_1}-\frac1{p_2})}_{p_2,r_2}.$
  \item The space  $\dot B^{\frac dp}_{p,1}$ is continuously embedded in   the set  of
bounded  continuous functions (going to $0$ at infinity if    $p<\infty$).
\item If $K$ is a smooth homogeneous of degree $m$ function   on $\R^d\setminus\{0\}$ 
that maps $\cS'_h$ to itself, then
\begin{equation}\label{R-E8}
K(D):\dot B^\sigma_{p,r}\to\dot B^{\sigma-m}_{p,r}.\end{equation}
 In particular, the gradient operator maps $\dot B^\sigma_{p,r}$
 to $\dot B^{\sigma-1}_{p,r}.$
  \end{itemize}
Let us also mention the following  interpolation inequality
 that is   satisfied whenever
   $1\leq p,r_1,r_2,r\leq\infty,$ $\sigma_1\not=\sigma_2$ and $\theta\in(0,1)$:
  \begin{equation}\label{R-E9}
  \|f\|_{\dot B^{\theta\sigma_2+(1-\theta)\sigma_1}_{p,r}}\lesssim\|f\|_{\dot B^{\sigma_1}_{p,r_1}}^{1-\theta}
  \|f\|_{\dot B^{\sigma_2}_{p,r_2}}^\theta.
  \end{equation}
The following proposition has been used in this paper.
\begin{prop}\label{prop2.1}
Let $\sigma\in \mathbb{R}$ and $1\leq p, r\leq \infty$. Let $(f_{j})_{j\in \mathbb{Z}} $ be a sequence of 
$L^p$ functions such that $\sum_{j\in\Z} f_j$ converges to some distribution $f$ in $\cS'_h$ and
$$\Bigl\|2^{j\sigma}\|f_j\|_{L^p(\R^d)}\Bigr\|_{\ell^r(\Z)}<\infty.
$$
If $\mathrm{Supp} \hat{f}_{j}\subset\mathcal{C}(0,2^jR_{1},2^jR_{2})$ for some $0<R_{1}<R_{2},$ then $f$ belongs to $\dot{B}^{\sigma}_{p,r}$ and there exists a constant $C$ such that
 \begin{equation}\label{R-E10}
\|f\|_{\dot{B}^{\sigma}_{p,r}}\leq C\Bigl\|2^{j\sigma}\|f_j\|_{L^p(\R^d)}\Bigr\|_{\ell^r(\Z)}\cdotp
 \end{equation}
 \end{prop}
The following result was used to bound  the terms of System \eqref{R-E1}
involving compositions of functions:
\begin{prop}\label{prop2.3}
Let $F:\R\rightarrow\R$ be  smooth with $F(0)=0.$
For  all  $1\leq p,r\leq\infty$ and  $\sigma>0$ we have
$F(f)\in \dot B^\sigma_{p,r}\cap L^\infty$  for  $f\in \dot B^\sigma_{p,r}\cap L^\infty,$  and
\begin{equation}\label{R-E12}
\|F(f)\|_{\dot B^\sigma_{p,r}}\leq C\|f\|_{\dot B^\sigma_{p,r}}
\end{equation}
with $C$ depending only on $\|f\|_{L^\infty},$ $F'$ (and higher derivatives),  $\sigma,$ $p$ and $d.$
\medbreak
If $\sigma>-\min(\frac dp,\frac d{p'}),$ then $f\in\dot B^\sigma_{p,r}\cap\dot B^{\frac dp}_{p,1}$
implies that $F(f)\in \dot B^\sigma_{p,r}\cap\dot B^{\frac dp}_{p,1},$ and
\begin{equation}\label{R-E13}
\|F(f)\|_{\dot B^\sigma_{p,r}}\leq C(1+\|f\|_{\dot B^{\frac dp}_{p,1}})\|f\|_{\dot B^\sigma_{p,r}}.
\end{equation}
\end{prop}
Let us finally recall the following classical \emph{Bernstein inequality}:
\begin{equation}\label{R-E14}
\|D^kf\|_{L^{b}}
\leq C^{1+k} \lambda^{k+d(\frac{1}{a}-\frac{1}{b})}\|f\|_{L^{a}}
\end{equation}
that holds  for all function $f$ such that
$\mathrm{Supp}\,\mathcal{F}f\subset \{\xi\in \mathbb{R}^{d}: |\xi|\leq R\lambda\}$ for some $R>0$
and $\lambda>0,$  if  $k\in\N$ and  $1\leq a\leq b\leq\infty$.
\smallbreak
Let us also recall that, as a consequence of \cite[Lemma 2.2]{BCD}, we have for all $s\in\R$ if
$\Supp\cF f\subset\{\xi\in\R^d: r\lambda\leq|\xi|\leq R\lambda\}$ for some $0<r<R,$ 
 \begin{equation}\label{R-E14b}
\|\Lambda^sf\|_{L^{b}} \approx  \lambda^{s}\|f\|_{L^{b}}\with \Lambda^s\triangleq(-\Delta)^{\frac s2}.
\end{equation}

When localizing PDE's by means of  Littlewood-Paley decomposition,  one  
ends up with  bounds for each dyadic block in spaces of type $L^{q}_{T}(L^{p})\triangleq L^q(0,T;L^p(\R^d)).$
To get a  Besov type information, we then have to  perform a summation on $\ell^{r}(\mathbb{Z}),$
which motivates  the following definition  that has been first introduced by J.-Y. Chemin in \cite{Chemin}
for $0\leq T\leq+\infty,$ $\sigma\in\mathbb{R}$ and  $1\leq p,q,r\leq\infty$:
$$\|f\|_{\widetilde{L}^{q}_{T}(\dot{B}^{\sigma}_{p,r})}\triangleq\Big\|\bigl(2^{j\sigma}\|\dot{\Delta}_{j}f\|_{L^{q}_{T}(L^{p})}\bigr)\Big\|_{\ell^r(\Z)}.$$
For notational simplicity, index $T$  is  omitted if $T=+\infty.$
\smallbreak
We also used the following functional space:
\begin{equation}\label{R-E15}\widetilde{\mathcal{C}}_{b}(\R_+;\dot{B}^{\sigma}_{p,r})\triangleq\bigl\{f\in \mathcal{C}(\R_+;\dot{B}^{\sigma}_{p,r})\ s.t.\
\|f\|_{\widetilde{L}^{\infty}(\dot{B}^{\sigma}_{p,r})}<\infty\bigr\}\cdotp
\end{equation}
The above norms  may be compared with those of the
more standard  Lebesgue-Besov spaces $L^{q}_{T}(\dot{B}^{\sigma}_{p,r})$ via
 Minkowski's inequality:
\begin{equation}\label{R-E16}
\|f\|_{\widetilde{L}^{q}_{T}(\dot{B}^{\sigma}_{p,r})}\leq\|f\|_{L^{q}_{T}(\dot{B}^{\sigma}_{p,r})}\,\,\,
\mbox{if }\,\, r\geq q,\ \ \ \
\|f\|_{\widetilde{L}^{q}_{T}(\dot{B}^{\sigma}_{p,r})}\geq\|f\|_{L^{\rho}_{T}(\dot{B}^{\sigma}_{p,r})}\,\,\,
\mbox{if }\,\, r\leq q.
\end{equation}
Restricting the above norms to the low or high frequencies parts of distributions is
fundamental in our approach. For some fixed  integer $k_0$ (the value of which  follows from the proof of the main theorem), we  put $z^\ell\triangleq \dot S_{k_0}z$ and $z^h\triangleq z-z^\ell,$ 
and\footnote{For technical reasons, we need a small
overlap between low and high frequencies.}
\begin{equation}\label{eq:k0}
\|z\|_{\dot B^s_{p,1}}^\ell\triangleq\sum_{k\leq k_0} 2^{ks}\|\ddk z\|_{L^p},\qquad
\|z\|_{\dot B^s_{p,1}}^h\triangleq\sum_{k\geq k_0-1} 2^{ks}\|\ddk z\|_{L^p},\end{equation}
$$
\|z\|_{\wt L^\infty_T(\dot B^s_{p,1})}^\ell\triangleq\sum_{k\leq k_0} 2^{ks}\|\ddk z\|_{L^\infty_T(L^p)}\quad\hbox{and}\quad
\|z\|_{\wt L^\infty_T(\dot B^s_{p,1})}^h\triangleq\sum_{k\geq k_0-1} 2^{ks}\|\ddk z\|_{L^\infty_T(L^p)}.
$$

\end{document}